\documentclass{amsart}%

\usepackage{amssymb}
\usepackage{amsmath}
\usepackage{amsthm}
\usepackage{amscd}
\usepackage{upgreek}
\usepackage[margin=1in]{geometry}
\usepackage{xcolor,tikz-cd, bbold}
\usepackage{csquotes}

\theoremstyle{plain}
\newtheorem{theorem}{Theorem}[section]
\newtheorem{proposition}{Proposition}[section]
\newtheorem{corollary}{Corollary}[section]
\newtheorem{lemma}{Lemma}[section]

\theoremstyle{remark}

\newtheorem{remark}{Remark}[section]
\newtheorem{examples}{Example}[section]
\newtheorem{assumption}{Assumption}[section]

\DeclareMathOperator{\dist}{dist}
\DeclareMathOperator{\supp}{supp}
\DeclareMathOperator{\diam}{diam}

\DeclareMathOperator{\lin}{span}
\DeclareMathOperator{\loc}{loc}

\DeclareMathOperator{\diag}{diag}
\DeclareMathOperator{\lip}{Lip}

\DeclareMathOperator{\Alt}{Alt}

\DeclareMathOperator{\sgn}{sgn}

\DeclareMathOperator{\im}{\mathrm{im}}

\begin{document}

\title{Differential complexes for local Dirichlet spaces, and non-local-to-local approximations}
\author{Michael Hinz$^1$}
\address{$^1$ Fakult\"{a}t f\"{u}r Mathematik, Universit\"{a}t Bielefeld, Postfach 100131, 33501 Bielefeld,
Germany}
\email{mhinz@math.uni-bielefeld.de}
\author{J\"orn Kommer$^2$}
\address{$^2$ Fakult\"{a}t f\"{u}r Mathematik, Universit\"{a}t Bielefeld, Postfach 100131, 33501 Bielefeld,
Germany}
\thanks{$^1$, $^2$ Research supported by the DFG IRTG 2235: \enquote{Searching for the regular in the irregular: Analysis of singular and random systems}. }
\email{joern.kommer@web.de}

\date{\today}

\begin{abstract}
We study differential $p$-forms on non-smooth and possibly fractal metric measure spaces, endowed with a local Dirichlet form. Using this local Dirichlet form, we prove a result on the localization of antisymmetric functions of $p+1$ variables on diagonal neighborhoods to differential $p$-forms. This result generalizes both the well-known classical localization on smooth Riemannian manifolds and the well-known semigroup approximation for quadratic forms. We observe that a related localization map taking functions into forms is well-defined and induces a chain map from a differential complex of 
Kolmogorov-Alexander-Spanier type onto a differential complex of deRham type.
\tableofcontents
\end{abstract}

\maketitle

\section{Introduction}

In this article we study differential forms and differential complexes based on local Dirichlet forms on metric measure spaces. Using the locality of the given form, we describe the passage from functions of $p+1$ variables on diagonal neighborhoods to differential $p$-forms by an explicit limit of \enquote{cotangential structures}. This limit relation is a natural link between differential complexes of Kolmogorov-Alexander-Spanier type and differential complexes of deRham type. Our main message is that a careful formulation of this limit relation will be general enough to apply to a wide variety of metric measure spaces, including products of fractal spaces.

Differential forms on smooth manifolds, \cite{dRh60}, are an established and widely studied subject. They are natural objects to be considered in differential geometry, integration theory and the study of topological invariants, \cite{BT82, dRh60, Warner}, and they have natural applications in mathematical physics, \cite{AHM88, Taylor, Temam97}. The study of differential forms on Riemannian manifolds from the point of view of variational calculus and elliptic operators, \cite{BL92, Dodziuk81} is central in Hodge theory, \cite{Cheeger, dRh60, dRh-K50, Gaffney55, Hodge41, Kodaira49}, and the study of geometric invariants, \cite{Gilkey74, Lueck, Pansu96, Rosenberg}. It also has important links to spectral theory, \cite{Rosenberg}, and numerical mathematics, \cite{Arnold18, Dodziuk76}. 

Dirichlet form theory, \cite{BH91, Da89, FOT94, LeJan78}, is a natural generalization of this variational point of view. It is well-developed at the level of functions, that is, \enquote{differential zero-forms}. Differential one-forms based on Dirichlet forms were introduced in \cite{CS03, Sauvageot89, Sauvageot90} and \cite[Sections 4 and 6]{W00} in a context of $C^\ast$- respectively Lipschitz algebras and in \cite[Chapter III]{Eberle99} in connection with abstract Sobolev spaces. A description of the main idea was also given in \cite[Chapter V, Exercise 5.9]{BH91} with an attribution to Mokobodzki, \cite[p. 305]{BH91}. The probabilistic counterpart of this approach seems to be even older, \cite{FOT94, Nakao85}. This Dirichlet form based concept was then used by several authors to study vector analysis, \cite{H15, HRT13}, and analysis on fractals, \cite{CGIS13, CS09, HKT15, HR16, HT15, HinzTeplyaev18, IRT12}. Except for \cite{CS03, H15, Sauvageot90} and the first sections in \cite{Sauvageot89} the primarily given objects in these sources were a local Dirichlet form or a diffusion generator. A first order derivation operator was then obtained by a subsequent construction. In this approach the Dirichlet form domain is Hilbert. As a particular consequence, also the resulting cotangent spaces are Hilbert. Except for \cite{HinzTeplyaev18} none of these references addressed differential $p$-forms for $p\geq 2$.

A strongly related approach to differential $p$-forms on metric measure spaces was investigated in \cite{Gigli17}. There a different starting point was chosen, with an emphasis on the metric structure and its interplay with the measure. The first order structure was introduced using minimal upper gradients and without the assumption that the domain of a related energy form would be a Hilbert space. Under the assumption that it is Hilbert, the resulting construction roughly speaking coincides with the one from the Dirichlet form approach, but it does not cover fractal spaces, see \cite[Section 2.3.5]{Gigli17}.
The later parts of \cite{Gigli17} contain a fairly comprehensive study of differential $p$-forms on $\mathrm{RCD}(K,\infty)$-spaces, including closed exterior derivations and Hodge theory, \cite[Section 3.5]{Gigli17}. An extension of this program to Dirichlet metric measure spaces with distribution-valued lower Ricci bounds is given in \cite{Braun21}.

Another established subject is the approach to cohomology by Kolmogorov, Alexander and Spanier, see
\cite{Alexander35, Kolmogorov36, Kolmogorov36a, Massey78, Spanier48, Spanier66, Warner}. In this theory differential complexes are formulated using germs of functions on products of the given space. The corresponding \enquote{derivation} is the usual coboundary operator, see (\ref{E:coboundary}) below. A more recent metric variant of this originally topological theory gained some attention, \cite{BSSS12, Genton, Pansu96, Pansu04, SS12}. It is related to the theory of Vietoris-Rips complexes, \cite{Gromov87, Hausmann95, Latschev01, Vietoris27}, which has links to interesting applications, \cite{Carlsson, ChazaldeSilvaOudot, ELZ}. The formulations \cite{BSSS12, Genton, SS12} used measures and integral kernels and  
 motivated our very recent study \cite{HinzKommer, Kommer} of differential complexes associated with generally unbounded non-local Dirichlet forms on metric measure spaces.

In this article we study how to pass from complexes of Kolmogorov-Alexander-Spanier type on metric measure spaces to complexes of, roughly speaking, deRham type. The \enquote{localization} of (antisymmetric) multivariate functions to differential forms can be implemented in many different ways. One is a purely algebraic factorization procedure, see \cite[Chapter 16]{Ei99}, \cite[Section 8.1]{GVF}. In this case it may not have any connection to the structure of the underlying space or to classical or functional analysis. Another way, which is typically used on smooth manifolds, is to differentiate along smooth curves, see for instance \cite[Section 1]{CM90} or \cite[Proposition 2.1]{SS12}; we recall details in Theorem \ref{T:localizationM} in Section \ref{S:manifolds} below. In this case one can define a surjective chain map $\lambda$ taking multivariate functions into differential forms. In view of common definitions, cf. \cite[Definition 1.13]{Warner} and \cite[Chapters 2 and 16]{Ei99}, this map $\lambda$ might be called a \emph{localization map}. Here we are interested in a formulation general enough to apply to fractal metric measure spaces, in which we cannot expect to have sufficiently many well-behaved curves. 

Since we assume that a local Dirichlet form is given, differential one-forms and first order derivations can be defined following 
the path of \cite[Chapter V, Exercise 5.9]{BH91} and \cite[Chapter III]{Eberle99} which uses a factorization by energy densities. See Theorem \ref{T:Eberle} and formulas (\ref{E:factorbyenergydens}) and (\ref{E:factorGamma}) below. This is in line with \cite{Gigli17} and with the formulations in \cite{CS03, Sauvageot89, Sauvageot90, W00}, which may be considered \enquote{integrated variants} of this idea. To have energy densities, that is, a carr\'e du champ, \cite{BH91}, we require the volume measure to be energy dominant, \cite{Hino10, HRT13}. In fractal examples this corresponds to the use of Kusuoka type measures, \cite{Kusuoka89}, respectively their products. Under mild further assumptions one can then take exterior products and construct a differential complex of deRham type. See Theorem \ref{T:localcomplex} for details.

Following \cite{HinzKommer,Kommer} we formulate complexes of Kolmogorov-Alexander-Spanier type in terms of antisymmetric functions of tensor product structure, Theorem \ref{T:elemcomplex}. This alone does not require anything further. If some additional hypotheses are satisfied, one can observe an explicit non-local-to-local limit of \enquote{cotangential structures}, Proposition \ref{P:limitx0}. In concrete applications this limit can be realized in different ways, some are described in Section \ref{S:realize}. A realization in terms of the associated Markov semigroup may be viewed as a determinantal generalization of the familiar semigroup approximation for Dirichlet forms, \cite[Lemma 1.3.4 (i)]{FOT94}, and as a heat flow based variant of localization. It is quite versatile and applies to Riemannian manifolds, Examples \ref{Ex:Rmf2}, to $\mathrm{RCD}^\ast(K,N)$-spaces, Examples \ref{Ex:RCDKN}, to degenerate diffusions, Examples \ref{Ex:degenerate},  and even to fractal spaces, Examples \ref{Ex:ProdSG2}. 

The mentioned non-local-to-local limit relation can be used to show the existence of a surjective chain map $\lambda$ from the Kolmogorov-Alexander-Spanier complex onto the deRham type complex, Theorem \ref{T:localization}. This result may be viewed as an energy based generalization of the aforementioned result in the manifold situation, Theorem \ref{T:localizationM}. 

We point out that a first and less general version of these results was obtained in the thesis \cite{Kommer}, on which this article is partially based. 

In Section \ref{S:KAS} we briefly recall basic definitions and observations around Kolmogorov-Alexander-Spanier complexes. In Section \ref{S:local} we recall basic facts on measurable fields of Hilbert spaces, local Dirichlet forms and related first order
structures and then formulate associated local differential complexes. In Section \ref{S:approx} we prove the mentioned non-local-to-local approximation for strongly local regular Dirichlet forms, suitably chosen kernels and a suitably chosen algebra of functions. Several realizations of this convergence and corresponding examples are discussed in Section \ref{S:realize}.

\section{Non-local complexes of multivariate functions}\label{S:KAS}

We recall some well-known definitions and elementary facts, \cite{Alexander35, BSSS12, CM90, Kolmogorov36, Kolmogorov36a, Massey78, SS12, Spanier48,Spanier66}, in the particular formulation used in \cite[Section 3]{HinzKommer}. 

Let $X$ be a set, $p\geq 1$ an integer, and let $\mathcal{S}_p$ denote the symmetric group of order $p$. A function $F:X^p\to \mathbb{R}$ is called \emph{antisymmetric} if $F(x_{\sigma(1)},...,x_{\sigma(p)})= \sgn(\sigma) F(x_1,...,x_p)$ for all permutations $\sigma\in \mathcal{S}_p$; here $\sgn \sigma$ denotes the sign of the permutation $\sigma$. We write 
\[\Alt_p(F)(x_1,...,x_p):=\frac{1}{p!}\sum_{\sigma\in \mathcal{S}_p}\sgn(\sigma)F(x_1,...,x_p)\]
for the \emph{antisymmetrizer} $\Alt_p$; it takes a function $F:X^p\to \mathbb{R}$ into an antisymmetric one.

If $F:X^p\to \mathbb{R}$ is a given function, then 
\begin{equation}\label{E:coboundary}
\delta_{p-1}F(x_0,...,x_p):=\sum_{i=0}^p(-1)^iF(x_0,...,\hat{x}_i,...,x_p),
\end{equation}
where $\hat{x}_i$ means that $x_i$ is omitted, defines a function $\delta_{p-1}F:X^{p+1}\to \mathbb{R}$. For any $p\ge 1$ the map $F\mapsto \delta_{p-1}F$ is linear, it is called the \emph{(Kolmogorov-Alexander-Spanier) coboundary operator} of order $p-1$. It satisfies 
\begin{equation}\label{E:deltadeltanull}
\delta_p\circ \delta_{p-1}=0,\quad p\geq 1,
\end{equation}
and
\begin{equation}\label{E:deltaaltcommute}
\delta_{p-1}\circ \Alt_p=\Alt_{p+1}\circ\: \delta_{p-1},\quad p\geq 1.
\end{equation}
Clearly $\delta_0\mathbf{1}=0$. 

Let $\mathcal{C}$ be a space of real valued functions on $X$. We write $\mathcal{C}^{0}:=\mathcal{C}$, and for $p\geq 1$ define
\begin{equation}\label{E:Cp}
\mathcal{C}^{p}:=\lin \{ \Alt_{p+1}(f_0\otimes f_1\otimes \dots\otimes f_p): f_0\in \mathcal{C}\oplus \mathbb{R}\ \text{and}\ f_1,\dots, f_p \in \mathcal{C}\}.
\end{equation}
To the elements of $\mathcal{C}^{p}$ we refer as \emph{elementary $p$-functions}. It is not difficult to see that for any $p\geq 1$ and $f_1,\dots, f_p \in \mathcal{C}$ we have 
\begin{equation}\label{E:deltaalttensor}
\delta_{p-1}\Alt_p(f_1\otimes \cdots\otimes f_p)=(p+1)\Alt_{p+1}(\mathbf{1}\otimes f_1\otimes \cdots\otimes f_p).
\end{equation}
Evaluated at $(x_0,x_1,...,x_p)\in X^{p+1}$ this equals
\begin{equation}\label{E:determinant}
\Alt_p(\delta_0 f_1(x_0,\cdot)\otimes \cdots \otimes \delta_0 f_p(x_0,\cdot))(x_1,...,x_p)=\frac{1}{p!}\det\big[(\delta_0 f_i(x_0,x_j))_{i,j=1}^p\big].
\end{equation}

Now suppose that $X$ is a topological space. We call a family $N_\ast=(N_p)_{p\geq 0}$ a \emph{system of diagonal neighborhoods} for $X$ if 
\begin{enumerate}
\item[(i)] the $N_p$, $p\geq 0$, are open neighborhoods of the diagonal $\diag_p:=\{(x_0,...,x_0): x_0\in X\}$ in $X^{p+1}$, respectively,
\item[(ii)] the $N_p$ are symmetric in the sense that for any $\pi\in \mathcal{S}_{p+1}$ and any $(x_{0},...,x_p)\in N_p$
we have $(x_{\pi(0)},...,x_{\pi(p)})\in N_p$, 
\item[(iii)] for any $p\geq 1$, any $(x_0,...,x_p)\in N_p$ and any $0\leq i\leq p$ we have $(x_0,...,\hat{x}_i,...,x_p)\in N_{p-1}$.
\end{enumerate}
See for instance \cite[Section 1]{CM90}.

\begin{remark}\mbox{}
\begin{enumerate}
\item[(i)] Note that we have $N_0=X$.  
\item[(ii)] For the purposes of this article it is convenient to define the neighborhoods $N_p$ as open sets.
\end{enumerate}
\end{remark}

\begin{examples}\label{Ex:metric}
If $(X,\varrho)$ is a metric space, then for any $\varepsilon>0$ the sets 
\begin{equation}\label{E:boundedrange}
N_p(\varepsilon):=\{(x_0,...,x_p)\in X^{p+1}: \max_{0\leq i<j\leq p}\varrho(x_i,x_j)<\varepsilon\}
\end{equation}
form a system $N_\ast(\varepsilon)=(N_p(\varepsilon))_{p\geq 0}$ of diagonal neighborhoods for $X$. 
\end{examples}

We set 
\begin{equation}\label{E:restrict}
\mathcal{C}^p(N_p):=\mathcal{C}^p|_{N_p},\quad p\geq 0.
\end{equation}
By (\ref{E:Cp}) and (\ref{E:deltaalttensor}) we have 
\[\delta_p:\mathcal{C}^p(N_p)\to \mathcal{C}^{p+1}(N_{p+1}),\quad p\geq 0.\]
Taking into account (\ref{E:deltadeltanull}), we obtain the following.

\begin{theorem}\label{T:elemcomplex} For any system $N_\ast=(N_p)_{p\geq 0}$ of diagonal neighborhoods the sequence
\begin{equation}\label{E:elemcomplex}
0\longrightarrow \mathcal{C}^0(N_0)\stackrel{\delta_0}{\longrightarrow}\mathcal{C}^1(N_1)\stackrel{\delta_1}{\longrightarrow} ... \stackrel{\delta_{p-1}}{\longrightarrow} \mathcal{C}^p(N_p)  \stackrel{\delta_p}{\longrightarrow} ...
\end{equation}
is a cochain complex. 
\end{theorem}

As usual, we write $(\mathcal{C}^\ast(N_\ast),\delta_\ast)$ with $\mathcal{C}^\ast(N_\ast)=(\mathcal{C}^p(N_p))_{p\geq 0}$ and $
\delta_\ast=(\delta_p)_{p\geq 0}$ for the complex in (\ref{E:elemcomplex}).
\begin{remark}
The complex $(\mathcal{C}^\ast(N_\ast),\delta_\ast)$ is \enquote{non-local} in the sense that the operators $\delta_p$ are non-local.
\end{remark}

\begin{remark}
If $\mathcal{A}$ is a vector subspace of $\mathcal{C}$, then the spaces $\mathcal{A}^p(N_p)$, analogously defined using (\ref{E:Cp}) and (\ref{E:restrict}), are subspaces of the spaces $\mathcal{C}^p(N_p)$,
we have 
\[\delta_p:\mathcal{A}^p(N_p)\to \mathcal{A}^{p+1}(N_{p+1}),\quad p\geq 0,\] 
and $(\mathcal{A}^\ast(N_\ast),\delta_\ast)$ is a subcomplex of $(\mathcal{C}^\ast(N_\ast),\delta_\ast)$.
\end{remark}

With the agreement that $\delta_{-1}:=0$, we write 
\begin{equation}\label{E:elemcoho}
H^p\mathcal{C}^\ast(N_\ast):=\ker \delta_p|_{\mathcal{C}^p(N_p)} / \im \delta_{p-1}|_{\mathcal{C}^p(N_p)}.
\end{equation}
for the $p$-th cohomology of (\ref{E:elemcomplex}), $p\geq 0$. 

\begin{remark} In the special case that $(X,\varrho)$ is a metric space and $N_\ast=N_\ast(\varepsilon)$ is as defined in (\ref{E:boundedrange}) the parameter $\varepsilon>0$ determines the metric scale at which these cohomologies can detect structural features of $X$, see for instance \cite{BSSS12, Genton, HinzKommer, SS12}. In particular, it is well-known that for $\varepsilon>\diam(X)$ the cohomologies $H^p\mathcal{C}^\ast(N_\ast(\varepsilon))$, $p\geq 1$, are all trivial. 
\end{remark}

We recall \cite[Proposition 3.2]{HinzKommer}, which is quickly seen.

\begin{proposition}\label{P:simplicial}
If $\mathcal{C}$ is an algebra, then for any $p\geq 1$ we have 
\begin{equation}\label{E:simplicial}
\mathcal{C}^{p}=\lin \{ \overline{g}\:\delta_{p-1}\Alt_p(f_1\otimes \dots\otimes f_p): g\in \mathcal{C}\oplus \mathbb{R} \ \text{and}\ f_1,\dots, f_p \in \mathcal{C}\},
\end{equation}
where 
\[\overline{g}(x_0,x_1,...,x_p):=\frac{1}{p+1}\sum_{i=0}^p g(x_i),\quad g\in \mathcal{C}\oplus \mathbb{R}.\]
Moreover, for any $g\in \mathcal{C}\oplus \mathbb{R}$ and $f_1, ..., f_p \in\mathcal{C}$ the identities
\begin{multline}\label{E:deltaact0}
(p+1)\bar{g}\delta_{p-1}\Alt_p(f_1\otimes \dots\otimes f_p)\\
=\Alt_{p+1}(g\otimes f_1\otimes \dots\otimes f_p)+\sum_{k=1}\Alt_{p+1}(\mathbf{1}\otimes f_1\otimes \cdots\otimes (f_kg)\otimes \cdots \otimes f_p)
\end{multline}
and 
\begin{equation}\label{E:deltaact}
\delta_p(\bar{g}\delta_{p-1}\Alt_p(f_1\otimes \dots\otimes f_p)) = \delta_p\Alt_{p+1}(g\otimes f_1\otimes \dots\otimes f_p)
\end{equation}
hold.
\end{proposition}

We also recall \cite[Lemma 6.2]{HinzKommer}. Similarly as before, writing $\hat{f}_\ell$ means that $f_\ell$ is omitted.
\begin{lemma} 
For any $p\geq 1$ and $f_0,...,f_p\in \mathcal{C}\oplus \mathbb{R}$ we have
\begin{equation}\label{E:pullout}
\Alt_{p+1}(f_0\otimes f_1\otimes \cdots \otimes f_p)=\frac{1}{p+1}\sum_{\ell=0}^p(-1)^\ell f_\ell\otimes \Alt_p(f_0\otimes \cdots\otimes \hat{f}_\ell\otimes \cdots\otimes f_p). 
\end{equation}
\end{lemma}

\section{Localization maps in the manifold case}\label{S:manifolds}

Suppose that $X=M$ is a smooth manifold, $\mathcal{C}=C_c^\infty(M)$ and $N_\ast=(N_p)_{p\geq 0}$ is a system of diagonal neighborhoods. As usual, the tangent and cotangent space at $x\in M$ are denoted by $T_xM$ and $T_x^\ast M$, respectively.

Given $p\geq 1$ and $x\in M$, we define a linear map $\lambda_{p,x}:\mathcal{C}^p(N_p)\to \Lambda^pT_x^\ast M$ by  
\begin{equation}\label{E:lambdapx}
\lambda_{p,x}(F):=\frac{\partial^p}{\partial t_1\cdots \partial t_p}F(x,\gamma_1(t_1),...,\gamma_p(t_p))|_{t_1=...=t_p=0},\quad F\in \mathcal{C}^p(N_p),
\end{equation}
for all $v_1,...,v_p\in T_x M$, where the $\gamma_i$ are smooth curves in $M$ such that $\gamma_i(0)=x$ and $\dot{\gamma}_i(0)=v_i$, $i=1,...,p$. This construction is classical and well known, see for instance \cite[Section 1]{CM90} or \cite[Section 2]{SS12}. Given $f\in \mathcal{C}^0(N_0)=C_c^\infty(M)$, let $\lambda_{0,x}(f):=f(x)$. 

\begin{remark}\label{R:trivial}
For $p>\dim M$ we have $\Lambda^pT_x^\ast M=\{0\}$ and $\lambda_{p,x}$ is just the zero map.
\end{remark}

Similarly as before, we slightly rearrange things. Let $d_{0,x}f$ denote the differential of $f\in \mathcal{C}$ at $x\in M$.

\begin{lemma} 
Let $p\geq 1$, $g\in \mathcal{C}\otimes \mathbb{R}$, $f_1,...,f_p\in\mathcal{C}$ and $x\in M$. Then we have 
\begin{equation}\label{E:niceformula}
\lambda_{p,x}(\overline{g}\delta_{p-1}\Alt_p(f_1\otimes \cdots\otimes f_p))=g(x)d_{0,x}f_1\wedge ...\wedge d_{0,x}f_p.
\end{equation}
\end{lemma}

\begin{proof}
From (\ref{E:lambdapx}) it is obvious that 
\begin{equation}\label{E:simpleformula}
\lambda_{p,x}(g\otimes \Alt_p(f_1\otimes \cdots \otimes f_p)=g(x)d_{0,x}f_1\wedge ...\wedge d_{0,x}f_p.
\end{equation}
By (\ref{E:deltaact0}) we find that 
\begin{multline}
(p+1)\lambda_{p,x}\big(\bar{g}\delta_{p-1}\Alt_p(f_1\otimes \dots\otimes f_p)\big)\notag\\
=\lambda_{p,x}\big(\Alt_{p+1}(g\otimes f_1\otimes \dots\otimes f_p)\big)+\sum_{k=1}^p\lambda_{p,x}\big(\Alt_{p+1}(\mathbf{1}\otimes f_1\otimes \cdots\otimes (f_kg)\otimes \cdots \otimes f_p)\big).
\end{multline}
By (\ref{E:pullout}) and (\ref{E:simpleformula}) the first summand on the right-hand side is seen to equal
\begin{multline}
\lambda_{p,x}\big(g\otimes \big(\Alt_p(f_1\otimes \dots\otimes f_p)\big)+\sum_{\ell=1}^p(-1)^\ell \lambda_{p,x}\big(f_\ell\otimes \Alt_p(g\otimes f_1\otimes \cdots \otimes \hat{f}_\ell\otimes \cdots \otimes f_p)\big)\notag\\
=g(x)d_{0,x}f_1\wedge ...\wedge d_{0,x}f_p+\sum_{\ell=1}^p (-1)^\ell f_\ell(x)d_{0,x}g\wedge d_{0,x}f_1\wedge ...\wedge \widehat{d_{0,x}f_\ell}\wedge ...\wedge d_{0,x}f_p\big).
\end{multline}
Using (\ref{E:lambdapx}) and the Leibniz rule for $d_0$, it follows that 
\begin{align}
&\lambda_{p,x}\big(\Alt_{p+1}(\mathbf{1}\otimes f_1\otimes \cdots \otimes (f_kg)\otimes\cdots \otimes f_p)\big)\notag\\
&=d_{0,x}f_1\wedge ...\wedge d_{0,x}(f_kg)\wedge ...\wedge d_{0,x}f_p\notag\\
&=g(x)d_{0,x}f_1\wedge ...\wedge d_{0,x}f_k\wedge ...\wedge d_{0,x}f_p-(-1)^kf_k(x)d_{0,x}g\wedge d_{0,x}f_1\wedge ...\wedge \widehat{d_{0,x}f_k}\wedge ...\wedge d_{0,x}f_p.\notag
\end{align}
Summing over $k=1,...,p$ and adding the sum to the above, we obtain (\ref{E:niceformula}).
\end{proof}

Setting $\lambda_p(F)(x):=\lambda_ {p,x}(F)$, $x\in M$, for any $F\in \mathcal{C}^p(N_p)$, we obtain linear maps 
\begin{equation}\label{E:lambdap}
\lambda_p:\mathcal{C}^p(N_p)\to \Omega_c^p(M)
\end{equation}
for all integers $p\geq 0$.

We write $\Omega^p_c(M)$ for the space of compactly supported smooth $p$-forms on $M$ and $d_p:\Omega^p_c(M)\to \Omega^{p+1}_c(M)$ for the exterior derivative acting on $\Omega^p_c(M)$. For $p=0$ this is consistent with the above notation $d_{0,x}$ in the sense that $d_{0,x}f$ equals $d_0f$, evaluated at $x$. For $p>\dim M$ the spaces $\Omega^p_c(M)$ are trivial, Remark \ref{R:trivial}. The \emph{deRham complex with compact supports} is denoted by $(\Omega_c^\ast(M),d_\ast)$, where $\Omega_c^\ast(M)=(\Omega_c^p(M))_{p\geq 0}$ and $d_\ast=(d_p)_{p\geq 0}$. We agree to set $d_{-1}:=0$ and with this agreement, write
\begin{equation}\label{E:deRhamcoho}
H^p\Omega^\ast(M):=\ker d_p|_{\Omega^p(M)} / \im d_{p-1}|_{\Omega^{p-1}(M)}
\end{equation}
for the \emph{$p$-th deRham cohomology with compact supports}, $p\geq 0$. See \cite[Section I.1]{BT82}. 

We recall the well-known link between the elementary complex and the deRham complex in the following variant.

\begin{theorem}\label{T:localizationM} Let $M$ be a smooth manifold, $\mathcal{C}=C_c^\infty(M)$ and let $N_\ast=(N_p)_{p\geq 0}$ be a system of diagonal neighborhoods.
\begin{enumerate}
\item[(i)] For any integer $p\geq 0$ the map $\lambda_p$ in (\ref{E:lambdap}) is a linear surjection.
\item[(ii)] The family $\lambda_\ast=(\lambda_p)_{p\geq 0}$ defines a cochain map $\lambda_\ast:(\mathcal{C}^\ast(N_\ast),\delta_\ast)\to (\Omega^\ast_c(M),d_\ast)$, that is, 
\[d_p\circ \lambda_p=\lambda_{p+1}\circ \delta_p,\quad p\geq 0.\]
In particular, it induces well-defined linear maps 
\[\lambda_p^\ast:H^p \mathcal{C}^\ast(N_\ast)\to H^p\Omega^\ast_c(M),\quad p\geq 0,\] 
between the cohomologies (\ref{E:elemcoho}) and (\ref{E:deRhamcoho}).
\end{enumerate}
\end{theorem}

\begin{proof}
Statement (i) is easily seen using a partition of unity and local coordinates, together with (\ref{E:restrict}), (\ref{E:simplicial}) and (\ref{E:niceformula}). To see statement (ii), it suffices to verify it for $F\in \mathcal{C}^p(N_p)$ of the form $F=\overline{g}\delta_{p-1}\Alt_p(f_1\otimes \cdots\otimes f_p)$. For such $F$ we observe that
\[d_p\circ \lambda_p(F)=d_p(gd_0f_1\wedge ...\wedge d_0f_p)=d_0g\wedge d_0f_1\wedge ...\wedge d_0f_p\]
and by (\ref{E:deltaact}) also
\[\lambda_{p+1}\circ \delta_p(F)=\lambda_{p+1}(\delta_p\Alt_{p+1}(g\otimes f_1\otimes \cdots\otimes f_p))=d_0g\wedge d_0f_1\wedge ...\wedge d_0f_p.\]
\end{proof}

The basic definition (\ref{E:lambdapx}) is based on differentiable curves. To define differential $p$-forms and localization maps $\lambda_p$ satisfying an analog of Theorem \ref{T:localizationM} on more non-smooth and possible fractal spaces, we investigate differential complexes based on local Dirichlet forms, \cite{BH91}, and then prove a non-local-to-local approximation, which is of independent interest. Identity (\ref{E:niceformula}) can then serve as a definition, and we can rely on the approximation result to ensure that this definition is correct. This point of view may be regarded as an energy based variant of (\ref{E:lambdapx}).

\section{Local complexes of differential forms}\label{S:local}

\subsection{Fields of Hilbert spaces}

We recall some terminology, notation and well-known facts around measurable fields of Hilbert spaces. Standard references are \cite{Dixmier81} and \cite{Takesaki79}.

Let $(X,\mathcal{X})$ be a measurable space. To a collection $H=(H_x)_{x\in X}$ of Hilbert spaces $(H_x, \left\langle\cdot,\cdot\right\rangle_{H_x})$ we refer as a \emph{field of Hilbert spaces on $X$}. The elements of the product $\prod_{x\in X} H_x$ are called \emph{sections of $H$}. 

A field of Hilbert spaces $H=(H_x)_{x\in X}$ on $X$, together with a subspace $\mathcal{M}(H)$ of the product $\prod_{x\in X} H_x$, is called a \emph{measurable field of Hilbert spaces on $X$}, if 
\begin{enumerate}
\item[(i)] a section $\omega=(\omega_x)_{x\in X}$ from $\prod_{x\in X} H_x$ is an element of $\mathcal{M}(H)$ if and only if the function $x\mapsto \left\langle \omega_x,\xi_x\right\rangle_{H_x}$ on $X$ is measurable for any $\xi=(\xi_x)_{x\in X}\in \mathcal{M}(H)$,
\item[(ii)] there is a countable set $\{\xi_i\}_{i=1}^\infty\subset \mathcal{M}(H)$ such that for all $x\in X$ the span of $\{\xi_{i,x}\}_{i=1}^\infty$ is dense in $H_x$.
\end{enumerate} 
See \cite[Part II, Chapter 1, Definition 1 and Remark 3]{Dixmier81}. The sections contained in $\mathcal{M}(H)$ are referred to as the \emph{measurable sections of $H$}. 

It is not difficult to see that if $(H_x)_{x\in X}$ is a measurable field of Hilbert spaces on $X$ and $A\in \mathcal{X}$, then $(H_x)_{x\in A}$, endowed with the space of measurable sections $\mathcal{M}(\widetilde{H}):=\{\xi|_A:\xi\in \mathcal{M}(H)\}$, is a measurable field of Hilbert spaces on $A$. See \cite[p. 165]{Dixmier81}.

In the sequel let $H$ be a measurable field of Hilbert spaces with space $\mathcal{M}(H)$ of measurable sections as above.

A countable set $\{\xi_i:\ i=1,2,...\}\subset \mathcal{M}(H)$ of measurable sections $\xi_i=(\xi_{i,x})_{x\in X}$ is called a \emph{measurable orthonormal frame for $H$} if at each $x\in X$ with $\dim H_x=+\infty$ the collection $\{\xi_x:\ i=1,2,...\}$ is
an orthonormal basis of $H_x$ and at each $x\in X$ with $d(x):=\dim H_x<+\infty$ the collection $\{\xi_{1,x},...,\xi_{d(x),x}\}$ is an orthonormal basis of $H_x$ and $\xi_{i,x}=0$, $i>d(x)$. A measurable orthonormal frame for $H$ always exists, \cite[Part II, Chapter 1, Proposition 1]{Dixmier81}.

It is easily seen that if $\omega=(\omega_x)_{x\in X}$ is an element of $\mathcal{M}(H)$ and $g$ is a measurable function on $X$, then $g\omega:=(g(x)\omega_x)_{x\in X}$ again defines an element of $\mathcal{M}(H)$. 

Now let $\mu$ be a $\sigma$-finite measure on $X$. Let $L^0(X,\mu)$ denote the space of $\mu$-equivalence classes of measurable functions on $X$ and $L^0(X,H,\mu)$ the space of $\mu$-equivalence classes of measurable sections of $H$. Then, given $g\in L^0(X,\mu)$ and $\omega\in L^0(X,H,\mu)$, the product $g\omega$ is again a well-defined element of $ L^0(X,H,\mu)$. Given $\omega,\xi\in L^0(X,H,\mu)$, we agree to 
\begin{equation}\label{E:agreement}
\text{write $\left\langle \omega,\xi\right\rangle_H$ for the well-defined $\mu$-equivalence class $x\mapsto \left\langle \omega_x,\xi_x\right\rangle_{H_x}$,}
\end{equation}
where $(\omega_x)_{x\in X}$ and $(\xi_x)_{x\in X}$ are arbitrary versions of the $\mu$-equivalence classes $\omega$ and $\xi$.

We say that two measurable fields of Hilbert spaces $H=(H_x)_{x\in X}$ and $\widetilde{H}=(\widetilde{H}_x)_{x\in X}$ on $X$ are \emph{essentially isometric} if there are a set $N\in \mathcal{X}$ with $\mu(N)=0$ and a collection $\Phi=\{\Phi_x\}_{x\in X\setminus N}$ of linear isometries $\Phi_x:H_x\to \widetilde{H}_x$, $x\in X\setminus N$, such that a section $\omega=(\omega_x)_{x\in X\setminus N}$ of 
$(H_x)_{x\in X\setminus N}$ is measurable if and only if $(\Phi_x(\omega_x))_{x\in X\setminus N}$ is a measurable section of $(\widetilde{H}_x)_{x\in X\setminus N}$. We then call $\Phi$ an \emph{essential isometry} from $H$ onto $\widetilde{H}$ \emph{outside $N$}. 

A measurable section $\omega=(\omega_x)_{x\in X}\in \mathcal{M}(H)$ is said to be square integrable with respect to $\mu$ if
\[\left\|\omega\right\|_{L^2(X,H,\mu)}^2:=\int_X\left\|\omega_x\right\|_{H_x}^2\mu(dx)<+\infty.\]
We write $(L^2(X,H,\mu), \left\|\cdot\right\|_{L^2(X,H,\mu)})$ for the \emph{direct integral of $H$}, that is, the Hilbert space of $\mu$-equivalence classes of square integrable sections. See \cite[Part II, Chapter 1, Section 5]{Dixmier81} or \cite[Chapter IV, Section 8]{Takesaki79}. If $\omega\in L^2(X,H,\mu)$ and $g\in L^\infty(X,\mu)$, then $g\omega\in L^2(X,H,\mu)$. If $\Phi$ is an essential isometry from $H$ onto $\widetilde{H}$ outside some set of zero measure, then $\Phi$ induces a unique linear isometry from $L^2(X,H,\mu)$ onto $L^2(X,\widetilde{H},\mu)$.

Given $1\leq r\leq +\infty$, we can introduce Banach spaces $L^r(X,T^\ast X,\mu)$ of $\mu$-equivalence classes of $r$-integrable (respectively essentially bounded) sections in a straightforward manner, see for instance \cite[Section 3]{HKM20} or  \cite[Section 6]{HRT13}.

\subsection{Local Dirichlet forms}\label{SS:Dirichletforms}

Let $(X,\mathcal{X},\mu)$ be a $\sigma$-finite measure space. By a \emph{quadratic form} $(\mathcal{E},\mathcal{D}(\mathcal{E}))$ on $L^2(X,\mu)$ we mean a densely defined and nonnegative definite symmetric bilinear form on $L^2(X,\mu)$. A quadratic form $(\mathcal{E},\mathcal{D}(\mathcal{E}))$ on $L^2(X,\mu)$ is said to be \emph{closed} if $\mathcal{D}(\mathcal{E})$, endowed with the norm 
\[f\mapsto (\mathcal{E}(f)+\|f\|_{L^2(X,\mu)}^2)^{1/2},\] 
is a Hilbert space. In general the map $\mapsto \mathcal{E}(f)^{1/2}$ is a only seminorm on $\mathcal{D}(\mathcal{E})$. A  quadratic form $(\mathcal{E},\mathcal{D}(\mathcal{E}))$ on $L^2(X,\mu)$ is said to have the \emph{Markov property} if for any 
Lipschitz function $F:\mathbb{R}\to \mathbb{R}$  with $F(0)=0$ and $f\in \mathcal{D}(\mathcal{E})$ we have $F(f)\in \mathcal{D}(\mathcal{E})$ and $\mathcal{E}(F(f))\leq \lip(F)^2\mathcal{E}(f)$. A \emph{Dirichlet form} $(\mathcal{E},\mathcal{D}(\mathcal{E}))$ on $L^2(X,\mu)$ is a closed quadratic form $(\mathcal{E},\mathcal{D}(\mathcal{E}))$ on $L^2(X,\mu)$ which has the Markov property. See \cite[Chapter I, Section 1]{BH91}.

Let $(\mathcal{E},\mathcal{D}(\mathcal{E}))$ be a Dirichlet form on $L^2(X,\mu)$ and let $b\mathcal{X}$ denote the space of bounded measurable functions. We write 
\begin{equation}\label{E:B}
\mathcal{B}:=b\mathcal{X}\cap \mathcal{D}(\mathcal{E})
\end{equation} 
for the space of bounded measurable functions on $X$ whose $\mu$-equivalence classes 
are elements of $\mathcal{D}(\mathcal{E})$ and, given $f\in \mathcal{B}$, understand the notation $\mathcal{E}(f)$
as the evaluation of $\mathcal{E}$ at the $\mu$-equivalence class of $f$. The space $\mathcal{B}$ is an algebra, \cite[Chapter I, Corollary 3.3.2]{BH91}.

We say that $(\mathcal{E},\mathcal{D}(\mathcal{E}))$ admits a \emph{carr\'e du champ} $\Gamma$ if for any $f\in \mathcal{B}$ there is some $\Gamma(f)\in L^1(X,\mu)$ such that 
\begin{equation}\label{E:carredef}
\mathcal{E}(fh,f)-\frac12\mathcal{E}(f^2,h)=\int_Xh\:\Gamma(f)\:d\mu,\quad h\in \mathcal{B},
\end{equation}
\cite[Chapter I, Definition 4.1.2]{BH91}. By polarization and approximation $\Gamma$ extends to a continuos nonnegative definite symmetric bilinear map $\Gamma$ from $\mathcal{D}(\mathcal{E})\times \mathcal{D}(\mathcal{E})$
into $L^1(X,\mu)$.

A Dirichlet form $(\mathcal{E},\mathcal{D}(\mathcal{E}))$ is called \emph{local} if for any $f\in \mathcal{D}(\mathcal{E})$ and any $F,G\in C_c^\infty(\mathbb{R})$ with disjoint supports we have $\mathcal{E}(F(f)-F(0),G(f)-G(0))=0$, \cite[Chapter I, Definition 5.1.2]{BH91}. If it is local and admits a carr\'e du champ $\Gamma$, then for any $f\in \mathcal{D}(\mathcal{E})$ we have 
\begin{equation}\label{E:totalenergy}
\mathcal{E}(f)=\int_X\Gamma(f)\:d\mu,
\end{equation}
\cite[Chapter I, Proposition 6.1.1]{BH91}. Under the same hypotheses a chain rule holds: Given $f_1,...,f_m, g_1,...,g_n\in \mathcal{D}(\mathcal{E})$, $F\in C^1(\mathbb{R}^m)$ with $F(0)=0$ and $G\in C^1(\mathbb{R}^n)$ with $G(0)=0$ and such that both $F$ and $G$ have uniformly bounded partial derivatives, the identity
\begin{equation}\label{E:chainrule}
\Gamma(F(f_1,...,f_m),G(g_1,...,g_n))=\sum_{i=1}^m\sum_{j=1}^n \frac{\partial F}{\partial x_i}(f_1,...,f_m)\frac{\partial G}{\partial x_j}(g_1,...,g_n)\Gamma(f_i,f_j)
\end{equation}
holds in the $\mu$-a.e. sense. See \cite{LeJan78} or the expositions in \cite[Chapter I, Corollary 6.1.3]{BH91} or \cite[Theorem 3.2.2]{FOT94}. In the special case $m=n=1$ this remains true for Lipschitz $F$ and $G$ with  $F(0)=0$ and $G(0)=0$, \cite[Chapter I, Corollary 7.1.2]{BH91}.  

The main assumption in this section is the following.

\begin{assumption}\label{A:Dirichletform}
We assume that $(X,\mathcal{X},\mu)$ is a $\sigma$-finite measure space, that $L^2(X,\mu)$ is separable and that $(\mathcal{E},\mathcal{D}(\mathcal{E}))$ is a local Dirichlet form on $L^2(X,\mu)$ admitting a carr\'e du champ $\Gamma$.  
\end{assumption}

Let Assumption \ref{A:Dirichletform} be in force and recall (\ref{E:B}). We call a pair $(H,d_0)$ a \emph{first order structure associated with $(\mathcal{E},\mathcal{D}(\mathcal{E}))$} if 
\begin{enumerate}
\item[(i)] $H=(H_x)_{x\in X}$ is a measurable field of Hilbert spaces on $X$,
\item[(ii)] $d_0$ is a linear map from $\mathcal{B}$ into $L^2(X,H,\mu)$ such that 
\begin{equation}\label{E:Gammaae}
\left\langle d_0f,d_0g\right\rangle_{H_x}=\Gamma(f,g)(x)\quad \text{for $\mu$-a.e. $x\in X$,}
\end{equation}
\item[(iii)] the space $\lin \{gd_0f: f\in \mathcal{B}\}$ is dense in $L^2(X,H,\mu)$.
\end{enumerate}
We call two first order structures $(H,d_0)$ and $(H',d_0')$ associated with $(\mathcal{E},\mathcal{D}(\mathcal{E}))$ \emph{essentially isometric} if there is an essential isometry $\Phi=\{\Phi_x\}_{x\in X\setminus N}$ from $H$ onto $H'$ outside  a $\mu$-null set $N\in \mathcal{X}$ and for any $f\in \mathcal{B}$ we have $d_0'f=\Phi(d_0f)$ in the $\mu$-a.e. sense.

Our starting point is the following variant of \cite[Chapter III, Theorem 3.11]{Eberle99}.  

\begin{theorem}\label{T:Eberle} Let Assumption \ref{A:Dirichletform} be in force. There is a first order structure $(T^\ast X,d_0)$ associated with $(\mathcal{E},\mathcal{D}(\mathcal{E}))$, and it is unique up to essential isometry. We have 
\begin{equation}\label{E:productrule}
d_0(fg)=fd_0g+gd_0f,\quad f,g\in \mathcal{B}.
\end{equation}
\end{theorem}

\begin{remark}\mbox{}
\begin{enumerate}
\item[(i)] The fibers $T^\ast_x X$ of $T^\ast X=(T_x^\ast X)_{x\in X}$ may be seen as \emph{generalized cotangent spaces} and the operator $d_0$ may be seen as a \emph{generalized exterior derivative} of order zero. 
\item[(ii)] Assumption \ref{A:Dirichletform} seems convenient since it covers a variety of applications. Under more restrictive hypotheses one can adopt a point of view similar to \cite[Definitions 3.1.1 and 3.1.3]{BGL14} and start from an algebra of functions and a local carr\'e du champ operator $\Gamma$ whose values are pointwise defined functions. One can then use \cite[Chapter III, Theorem 3.9]{Eberle99} to obtain less technical variants of Theorem \ref{T:Eberle} and subsequent results. This has been implemented in \cite{Kommer}. 
\item[(iii)] Under the stated assumptions the operator $d_0$ extends to a closed unbounded linear operator $d_0$ from 
$L^2(X,\mu)$ into $L^2(X,T^\ast X,\mu)$ with domain $\mathcal{D}(\mathcal{E})$, and the closedness of the operator $(d_0,\mathcal{D}(\mathcal{E}))$ is equivalent to the closedness of $(\mathcal{E},\mathcal{D}(\mathcal{E}))$. Up to an essential isometry of first order structures, the direct integral $L^2(X, T^\ast X,\mu)$ and the operator $d_0$ agree with the Hilbert space and the derivation introduced in \cite{CS03}, see for instance \cite[Theorem 2.1 and Corollary 2.5]{HRT13}. A related approach was investigated in \cite{W00}, see in particular \cite[Definition 52, Proposition 57 and Theorem 60]{W00}.
\end{enumerate}
\end{remark}

An outline of the proof of Theorem \ref{T:Eberle} can be found in \cite[Chapter V, Exercise 5.9]{BH91}, a complete proof is given in \cite[p. 164--167]{Eberle99}. To prepare notation for later arguments, we recall some key elements of this proof: Under Assumption \ref{A:Dirichletform} the space $\mathcal{D}(\mathcal{E})$ is separable with respect to the seminorm $\mathcal{E}^{1/2}$. Using the Markov property, we can find a sequence $(\varphi_k)_{k\geq 1}\subset \mathcal{B}$ whose span 
\begin{equation}\label{E:spaceA}
\mathcal{A}:=\lin (\{\varphi_k\}_{k\geq 1})
\end{equation}
is $\mathcal{E}^{1/2}$-dense in $\mathcal{B}$, see \cite[Theorem 1.4.2]{FOT94}. Let $x\mapsto \widetilde{\Gamma}(\varphi_i,\varphi_j)(x)$ be versions of the $\mu$-equivalence classes $\Gamma(\varphi_i,\varphi_j)$ such that for any $N\in \mathbb{N}$ and $x\in X$ the matrix $(\widetilde{\Gamma}(\varphi_i,\varphi_j))_{i,j=1}^N$ is symmetric and non-negative definite over $\mathbb{Q}^N$ and consequently also over $\mathbb{R}^N$. Then for any $u,v\in \mathcal{A}$ the map $x\mapsto \widetilde{\Gamma}(u,v)(x)$ is well-defined as a version of $\Gamma(u,v)$, and for each $x\in X$ the map $(u,v)\mapsto \widetilde{\Gamma}(u,v)(x)$ is a nonnegative definite symmetric bilinear form on $\mathcal{A}$. Now let $x\in X$ be fixed. The definition
\begin{equation}\label{E:factorbyenergydens}
\left\|u\right\|_{T_x^\ast X}:=\big(\widetilde{\Gamma}(u,u)(x)\big)^{1/2},\quad u\in \mathcal{A},
\end{equation}
gives a pre-Hilbert norm on the quotient space $\mathcal{A}/\ker \left\|\cdot\right\|_{T_x^\ast X}$. We write $(T_x^\ast X,\left\langle \cdot,\cdot\right\rangle_{T_x^\ast X})$ for the Hilbert space obtained as the completion of the quotient with respect to 
$\left\|\cdot\right\|_{T_x^\ast X}$. Given $f\in \mathcal{A}$, we denote its equivalence class in $\mathcal{A}/\ker \left\|\cdot\right\|_{T_x^\ast X}$ by $d_{0,x}f$. It follows that 
\begin{equation}\label{E:factorGamma}
\left\langle d_{0,x}f,d_{0,x}g\right\rangle_{T_x^\ast X}=\widetilde{\Gamma}(f,g)(x),\quad f,g\in \mathcal{A}.
\end{equation}
This defines a field 
\[T^\ast X:=(T_x^\ast X)_{x\in X}\] 
of Hilbert spaces on $X$. When endowed with the space $\mathcal{M}(T^\ast X)$ of all sections $\omega\in \prod_{x\in X} T_x^\ast X$ for which all functions $x\mapsto \left\langle \omega_x, d_{0,x}f\right\rangle_{T_x^\ast M}$, $f\in \mathcal{A}$, are measurable, $T^\ast X$ becomes a measurable field of Hilbert spaces on $X$. This follows from \cite[Part II, Chapter 1, Section 4, Proposition 4]{Dixmier81}. Since for any $f\in \mathcal{A}$ the section $d_0f:=(d_{0,x}f)_{x\in X}$ is in $\mathcal{M}(T^\ast X)$ and
\begin{equation}\label{E:differentialsandenergy}
\left\|d_0f\right\|_{L^2(X,T^\ast X,\mu)}^2=\mathcal{E}(f)
\end{equation}
by (\ref{E:totalenergy}) and (\ref{E:factorGamma}), the map $f\mapsto d_0f$ extends uniquely to a linear map $d_0$ from 
$\mathcal{B}$ into $L^2(X,T^\ast X,\mu)$ which satisfies (\ref{E:differentialsandenergy}) for all $f\in \mathcal{B}$. Taking limits in (\ref{E:factorGamma}) gives (\ref{E:Gammaae}). The density of $\lin\{gd_0f:f,g\in \mathcal{B}\}$ in $L^2(X,T^\ast X,\mu)$ is easily seen using approximation and a totality argument, the
product rule (\ref{E:productrule}) follows from (\ref{E:chainrule}). The proof of uniqueness up to essential isometry uses (\ref{E:Gammaae}), the separability of $\mathcal{B}$ and the aforementioned density, see \cite[p. 166/167]{Eberle99}. 

\begin{remark}
As a particular consequence of Theorem \ref{T:Eberle}, the pair $(T^\ast X, d)$, determined up to essential isometry, is independent of the choice of the sequence $(\varphi_k)_{k\geq 1}$.
\end{remark}

\subsection{Local complexes}

In what follows let Assumption \ref{A:Dirichletform} be in force, let $(T^\ast X,d_0)$ be a first order structure associated with $(\mathcal{E},\mathcal{D}(\mathcal{E}))$ as in Theorem \ref{T:Eberle}, and let $\mathcal{A}$ be as in (\ref{E:spaceA}). 

For any fixed $x\in X$ and integer $p\geq 1$ let $\Lambda^p T_x^\ast X$ be the $p$-fold exterior product of the vector space $T_x^\ast X$. We endow $\Lambda^p T_x^\ast X$ with the scalar product
\begin{equation}\label{E:extprodspace}
\left\langle v_1\wedge ... \wedge v_p, w_1\wedge ... \wedge w_p\right\rangle_{\Lambda^p T_x^\ast X}:=\det[(\left\langle v_i,w_j\right\rangle_{T_x^\ast X})_{i,j=1}^p],
\end{equation}
$v_i,w_j\in T_x^\ast X$, and write $\hat{\Lambda}^pT_x^\ast X$ for the completion of $\Lambda^pT_x^\ast X$ with respect to $\left\langle\cdot,\cdot\right\rangle_{\Lambda^pT_x^\ast X}$, see for instance \cite[Chapter V, Section 1]{Temam97}. Obviously $\hat{\Lambda}^1T_x^\ast X=T_x^\ast X$. 

Using (\ref{E:extprodspace}) it is straightforward to see that $\lin\{d_{0,x}f_1\wedge ... \wedge d_{0,x}f_p:\ f_1,...,f_p\in \mathcal{A}\}$ is dense in $\Lambda^p T_x^\ast X$ and therefore also in $\hat{\Lambda}^p T_x^\ast X$. Set $\mathcal{M}^1(T^\ast X):=\mathcal{M}(T^\ast X)$ and for $p\geq 2$, let $\mathcal{M}^p(T^\ast X)$ denote the space of all elements $\omega=(\omega_x)_{x\in X}$ of $\prod_{x\in X} \hat{\Lambda}^p T_x^\ast X$ for which all functions $x\mapsto \left\langle \omega_x,d_{0,x}f_1\wedge ...\wedge d_{0,x}f_p\right\rangle_{\Lambda^p T_x^\ast X}$, $f_1,...,f_p\in\mathcal{A}$, are measurable. For any $p\geq 1$ we set
\begin{equation}\label{E:OmegapA}
\Omega^p(\mathcal{A}):=\lin\left\lbrace g\:d_0f_1\wedge ...\wedge d_0f_p:\ f_1,...,f_p\in \mathcal{A},\ g\in \mathcal{A}\oplus \mathbb{R}\right\rbrace.
\end{equation}

\begin{lemma}\label{L:Lambdameasfield}
Let Assumption \ref{A:Dirichletform} be in force. For any $p\geq 1$ and $x\in X$ the collection $\hat{\Lambda}^p T^\ast X=(\hat{\Lambda}^p T_x^\ast X)_{x\in X}$, endowed with $\mathcal{M}^p(T^\ast X)$ as space of measurable sections, is a measurable field of Hilbert spaces on $X$. Moreover, $\Omega^p(\mathcal{A})\subset \mathcal{M}^p(T^\ast X)$.
\end{lemma}

\begin{proof}
For any $f_1,...,f_p,g_1,...,g_p\in \mathcal{A}$ the function 
\[x\mapsto \left\langle d_{0,x}f_1\wedge...\wedge d_{0,x}f_p, d_{0,x}g_1\wedge...\wedge d_{0,x}g_p\right\rangle_{\Lambda^pT_x^\ast X}\]  
is measurable by (\ref{E:factorGamma}) and (\ref{E:extprodspace}). Using \cite[Part II, Chapter 1, Section 4, Proposition 4 and its proof]{Dixmier81} we obtain the first statement. The second statement is immediate.
\end{proof}

For an element $\omega=(\omega_x)_{x\in X}$ of $\Omega^p(\mathcal{A})$ of the form $\omega=g \:d_0f_1\wedge ...\wedge d_0f_p$ 
we have the pointwise representation
\begin{equation}\label{E:pointwiseA}
\omega_x=g(x) \:d_{0,x}f_1\wedge ...\wedge d_{0,x}f_p,\quad x\in X.
\end{equation}
We write 
\begin{equation}\label{E:OmegapAx}
\Omega^p(\mathcal{A})_x:=\{\omega_x:\ \omega\in \Omega^p(\mathcal{A})\}.
\end{equation}

Our next goal is to \enquote{replace} the vector space $\mathcal{A}$ in (\ref{E:OmegapA}) by a suitable subalgebra $\mathcal{C}\supset \mathcal{A}$ of $\mathcal{B}$. Some preparations are needed.

Given $v=v_1\wedge ...\wedge v_p\in \Lambda^p T_x^\ast X$ and $w=w_1\wedge ...\wedge w_q\in \Lambda^q T_x^\ast X$ with $v_i,w_j\in T_x^\ast X$, their wedge product 
\begin{equation}\label{E:wedgeprodplain}
v\wedge w:=v_1\wedge ...\wedge v_p\wedge w_1\wedge ...\wedge w_q
\end{equation}
is an element of $\Lambda^{p+q} T_x^\ast X$. We observe the following properties of the wedge product, a proof is given in Appendix \ref{App:wedge}.

\begin{lemma}\label{L:wedgewelldef} Let Assumption \ref{A:Dirichletform} be in force and $p,q\geq 1$.
\begin{enumerate}
\item[(i)] For any $x\in X$, $v\in \Lambda^p T^\ast_x X$ and $w\in \Lambda^q T_x^\ast X$ we have 
\begin{equation}\label{E:wedgebound}
\left\|v\wedge w\right\|_{\Lambda^{p+q}T_x^\ast X}\leq \frac{(p+q)!}{p!q!}\left\|v\right\|_{\Lambda^{p}T_x^\ast X}\left\|w\right\|_{\Lambda^{q}T_x^\ast X},
\end{equation}
and (\ref{E:wedgeprodplain}) extends to a unique bilinear map $\wedge: \hat{\Lambda}^{p}T_x^\ast X\times \hat{\Lambda}^{q}T_x^\ast X\to \hat{\Lambda}^{p+q}T_x^\ast X$ satisfying the same estimate.
\item[(ii)] Given $\omega\in \mathcal{M}^p(T^\ast X)$ and $\eta\in \mathcal{M}^q(T^\ast X)$, the section 
\begin{equation}\label{E:wedgeviaversions}
\omega\wedge \eta:=(\omega_x\wedge \eta_x)_{x\in X}
\end{equation}
is measurable,  $\omega\wedge \eta\in \mathcal{M}^{p+q}(T^\ast X)$.
\item[(iii)] Given $\omega\in L^0(\hat{\Lambda}^pT^\ast X,\mu)$ and $\eta\in L^0(\hat{\Lambda}^qT^\ast X,\mu)$, their wedge product $\omega\wedge \eta$, defined using arbitrary versions in (\ref{E:wedgeviaversions}), is a well-defined element of  $L^0(\hat{\Lambda}^{p+q}T^\ast X,\mu)$.
\item[(iv)] Given $\omega\in L^2(\hat{\Lambda}^pT^\ast X,\mu)$ and $\eta\in L^\infty(\hat{\Lambda}^qT^\ast X,\mu)$, we have $\omega\wedge \eta\in L^2(\hat{\Lambda}^{p+q}T^\ast X,\mu)$, and similarly with the roles of $\omega$ and $\eta$ swapped.
\end{enumerate}
\end{lemma}

We make the following additional assumption.

\begin{assumption}\label{A:algebraC}\mbox{}
We assume that $\mathcal{C}$ is an $\mathcal{E}^{1/2}$-dense subalgebra of $\mathcal{B}$ and $\Gamma(f)\in L^\infty(X,\mu)$ for all $f\in \mathcal{C}$.
\end{assumption}

Assumption \ref{A:algebraC} is a sufficient condition to guarantee that Lemma \ref{L:wedgewelldef} (iv) applies to all differentials $df$, $f\in \mathcal{C}$. Under Assumption \ref{A:algebraC} we may and do assume that 
\begin{equation}\label{E:spandense}
\text{$(\varphi_k)_{k\geq 1}\subset \mathcal{C}$ is such that $\mathcal{A}=\lin(\{\varphi_k\}_{k\geq 1})$ is dense in $\mathcal{B}$.}
\end{equation}
Indeed, $\mathcal{D}(\mathcal{E})$ being $\mathcal{E}^{1/2}$-separable, also its subset $\mathcal{C}$ is $\mathcal{E}^{1/2}$-separable, so that we can find $(\varphi_k)_{k\geq 1}\subset \mathcal{C}$ with span $\mathcal{A}$ being $\mathcal{E}^{1/2}$-dense in $\mathcal{C}$ and therefore also $\mathcal{E}^{1/2}$-dense in $\mathcal{D}(\mathcal{E})$.

Let Assumption \ref{A:algebraC} be satisfied. For any $p\geq 1$ we set
\begin{equation}\label{E:OmegapC}
\Omega^p(\mathcal{C}):=\lin\left\lbrace g\:d_0f_1\wedge ...\wedge d_0f_p:\ f_1,...,f_p\in \mathcal{C},\ g\in \mathcal{C}\oplus \mathbb{R}\right\rbrace.
\end{equation}
An inductive application of Lemma \ref{L:wedgewelldef} (iv) justifies definition (\ref{E:OmegapC}) and implies that
\begin{equation}\label{E:inclusion}
\Omega^p(\mathcal{C})\subset L^2(\hat{\Lambda}^pT^\ast X,\mu)\cap L^\infty(\hat{\Lambda}^pT^\ast X,\mu),\quad  p\geq 1.
\end{equation}
We refer to the elements of $\Omega^p(\mathcal{C})$ as \emph{elementary $p$-forms associated with the algebra $\mathcal{C}$}.
Clearly 
\[\Omega^p(\mathcal{A})\subset \Omega^p(\mathcal{C}).\]

\begin{lemma}\label{L:integratedOmegas}
Suppose that Assumptions \ref{A:Dirichletform} and \ref{A:algebraC} hold and that $\mathcal{A}$ is as in (\ref{E:spandense}). Then for any $p\geq 1$ the space $\Omega^p(\mathcal{A})$ is dense in $L^2(X,\hat{\Lambda}^pT^\ast X,\mu)$, and consequently the same is true for $\Omega^p(\mathcal{C})$.
\end{lemma}

\begin{proof}
Suppose that $\omega\in L^2(X,\hat{\Lambda}^pT^\ast X,\mu)$ is such that 
\[\int_Xg(x)\left\langle \omega_x,d_{0,x}f_1\wedge ...\wedge d_{0,x}f_p\right\rangle_{\Lambda^pT_x^\ast X}\mu(dx)=\left\langle \omega,gd_0f_1\wedge ...\wedge d_0f_p\right\rangle_{L^2(X,\hat{\Lambda}^pT^\ast X,\mu)}=0\]
for all $f_1,...,f_p\in \mathcal{A}$ and $g\in \mathcal{A}\oplus \mathbb{R}$. This is possible only if for any fixed $f_1,...,f_p\in \mathcal{A}$ the signed measure
$\left\langle \omega_x,d_{0,x}f_1\wedge ...\wedge d_{0,x}f_p\right\rangle_{\Lambda^pT_x^\ast X}\mu(dx)$ is the zero measure, note that since $\mathcal{A}$ is dense in $L^2(X,\mu)$, indicators of sets of finite measure can be approximated pointwise $\mu$-a.e. by functions from $\mathcal{A}$. This can happen only if 
\begin{equation}\label{E:nullatx}
\left\langle \omega_x,d_{0,x}f_1\wedge ...\wedge d_{0,x}f_p\right\rangle_{\Lambda^pT_x^\ast X}=0
\end{equation}
at $\mu$-a.e. $x\in X$. We can find a set $Z\in \mathcal{X}$ of measure zero such that (\ref{E:nullatx}) holds for all $f_1,...,f_p\in \lin_\mathbb{Q}(\{\varphi_k\}_{k\geq 1})$ and all $x\in X\setminus Z$. But this implies that (\ref{E:nullatx}) holds for all $f_1,...,f_p\in \mathcal{A}$ and all $x\in X\setminus Z$, therefore $\omega_x=0$ at all such $x$ and consequently $\omega=0$ in $L^2(X,\hat{\Lambda}^pT^\ast X,\mu)$.
\end{proof}

We introduce exterior derivatives of higher order by a mimicry of the classical definition. Let Assumption \ref{A:algebraC} be satisfied. If $\mathcal{C}$ contains the constants, then by locality we have 
\begin{equation}\label{E:anniconst}
d_0\mathbf{1}=0;
\end{equation}
in this case we write $\Omega^0(\mathcal{C}):=\mathcal{C}$. If $\mathcal{C}$ does not include the constants, we write $\Omega^0(\mathcal{C}):=\mathcal{C}\oplus \mathbb{R}$ and extend $d_0$ to a linear operator $d_0:\Omega^0(\mathcal{C})\to \Omega^1(\mathcal{C})$ by setting
\[d_0(f+c):=d_0f,\quad f\in \mathcal{C},\ c\in \mathbb{R};\] 
this implies (\ref{E:anniconst}). For $p\geq 1$, $f_1,...,f_p\in\mathcal{C}$ and $g\in \mathcal{C}\oplus \mathbb{R}$ we define
\[d_p(g\:d_0f_1\wedge ...\wedge d_0f_p):=d_0g\wedge d_0f_1\wedge ... \wedge d_0f_p\]
and extend $d_p$ to a linear map $d_p:\Omega^p(\mathcal{C})\to \Omega^{p+1}(\mathcal{C})$. It satisfies 
$d_p\circ d_{p-1}=0$. This gives the following observation.

\begin{theorem}\label{T:localcomplex} Let Assumptions \ref{A:Dirichletform} and \ref{A:algebraC} be in force. Then the sequence
\begin{equation}\label{E:localcomplex}
0\longrightarrow \Omega^0(\mathcal{C})\stackrel{d_0}{\longrightarrow}\Omega^1(\mathcal{C})\stackrel{d_1}{\longrightarrow} ... \stackrel{d_{p-1}}{\longrightarrow} \Omega^p(\mathcal{C})  \stackrel{d_p}{\longrightarrow} ...
\end{equation}
is a cochain complex. 
\end{theorem}

Similarly as before, we write  $(\Omega^\ast(\mathcal{C}),d_\ast)$ for the complex in (\ref{E:localcomplex}).

\begin{remark}
The complex $(\Omega^\ast(\mathcal{C}),d_\ast)$ is \enquote{local} in the sense that the operators $d_p$ are local.
\end{remark}

With the agreement that $d_{-1}:=0$, we write 
\begin{equation}\label{E:localcoho}
H^p\Omega^\ast(\mathcal{C}):=\ker d_p|_{\Omega^p(\mathcal{C})} / \im d_{p-1}|_{\Omega^{p-1}(\mathcal{C})}.
\end{equation}
for the $p$-th cohomology of (\ref{E:localcomplex}), $p\geq 0$.

\begin{examples}\label{Ex:Rmf} 
Suppose that $X=M$ is a smooth $n$-dimensional Riemannian manifold and $\mu$ is a measure on $M$ having a strictly positive smooth density with respect to the Riemannian volume. Then $(M,\mu)$ is referred to as a \emph{weighted manifold}, \cite[Definition 3.17]{Grigoryan2009}. Consider the Dirichlet integral
\begin{equation}\label{E:Dirichletintegral}
\mathcal{E}(f)=\int_M\left\| \nabla f\right\|_{TM}^2\mu(dx),\quad f\in \mathcal{D}(\mathcal{E}),
\end{equation}
where $\mathcal{D}(\mathcal{E}):=W_0^1(M)$ is the closure of $C_c^\infty(M)$ in the space $W^1(M)$ of all $f\in L^2(M)$ with $\left\|\nabla f\right\|_{TM}$ in $L^2(M,\mu)$ and with norm $f\mapsto (\left\|f\right\|_{L^2(M,\mu)}^2+\left\|\nabla f\right\|_{L^2(M,\mu)}^2)^{1/2}$. Then Assumption \ref{A:Dirichletform} holds, note that $\Gamma(f)=\|\nabla f\|_{TM}^2$. With $\mathcal{C}=C_c^\infty(M)$ also Assumption \ref{A:algebraC} is satisfied. At $\mu$-a.e. $x\in M$ the spaces $T_x^\ast M$ in Theorem \ref{T:Eberle}, seen as vector spaces, coincide with the classical cotangent spaces. In particular, we have $\Lambda^p T_x^\ast M=\{0\}$, $p>n$, at $\mu$-a.e. $x\in M$, in line with Remark \ref{R:trivial}. The spaces $\Omega^p(C_c^\infty(M))$ are the spaces $\Omega_c^p(M)$ of compactly supported smooth $p$-forms, $d_\ast$ is the exterior derivative and (\ref{E:localcomplex}) and (\ref{E:localcoho}) coincide with the deRham complex and the deRham cohomologies with compact supports (\ref{E:deRhamcoho}).
\end{examples}

\begin{examples}\label{Ex:CD} Suppose that $(X,\varrho)$ is a complete separable metric space endowed with a Borel probability measure $\mu$ having finite second moment. Let $(\mathcal{E},\mathcal{D}(\mathcal{E}))$ be the Cheeger energy as defined in \cite[formula (1.1)]{AGS14}, and assume it is quadratic. Then it is a local Dirichlet form, and the algebra $\mathcal{C}:=\lip_b(X)$ of bounded Lipschitz functions is dense in $\mathcal{D}(\mathcal{E})$, \cite[Proposition 4.10]{AGS14}. Both Assumptions \ref{A:Dirichletform} and \ref{A:algebraC} are satisfied. In this situation $\Gamma(f)=|\nabla f|^2_\ast$, where $|\nabla f|_\ast$ denotes the minimal relaxed gradient of $f\in \mathcal{D}(\mathcal{E})$; it is well known that under the stated hypotheses it can be replaced by Cheeger's gradient, \cite{Ch99}, or the minimal upper gradient of $f$, \cite{Sh00}. See also \cite{KoskelaZhou}.

Under the additional assumption that $(X,\varrho,\mu)$ satisfies the $\mathrm{RCD}(K,\infty)$-condition with some $K\in \mathbb{R}$
a theory of differential complexes was studied in \cite[Section 3.5]{Gigli17}. There the closability of exterior derivations, \cite[Theorem 3.5.2 ii)]{Gigli17}, was ensured by the lower Ricci curvature bound, \cite[Remark 3.5.3]{Gigli17}, and a version of the Hodge theorem was proved, \cite[Theorem 3.5.15]{Gigli17}.
\end{examples}

\begin{remark} In general the exterior derivations $d_p$ of order $p\geq 1$ may not necessarily be closable, see for instance \cite{HinzTeplyaev18} for a counterexample.
\end{remark}

\subsection{Comments on regularity and change of measure}\label{SS:stronglylocal}

Recall that if $(X,\varrho)$ is a locally compact separable metric space and $\mu$ is a nonnegative Radon measure on $X$ with full support, then a Dirichlet form $(\mathcal{E},\mathcal{D}(\mathcal{E}))$ on $L^2(X,\mu)$ is said to be \emph{regular}, \cite{FOT94}, if $C_c(X)\cap\mathcal{D}(\mathcal{E})$ is dense both in $\mathcal{D}(\mathcal{E})$ and in $C_c(X)$. A regular Dirichlet form $(\mathcal{E},\mathcal{D}(\mathcal{E}))$ is said to be \emph{strongly local} if 
$\mathcal{E}(f,g)=0$
for all $f,g\in C_c(X)\cap\mathcal{D}(\mathcal{E})$ such that $g$ is constant on a neighborhood of $\supp f$. Strong locality implies locality in the sense of Subsection \ref{SS:Dirichletforms}. 

One way to satisfy Assumptions  \ref{A:Dirichletform} and \ref{A:algebraC} is to use the interplay of suitable \enquote{coordinates} and specific measures. This idea is well known, and it is robust enough to cover even very non-classical situations, as the following example shows.

\begin{examples}\label{Ex:ProdSG}
Let $K\subset \mathbb{R}^2$ denote the Sierpinski gasket and $(\mathcal{E},\mathcal{F})$ the standard resistance form on $K$, \cite{Ki89, Ki01, Str06}. We use the associated resistance metric on $K$. Let $h_1, h_2\in \mathcal{F}$ be harmonic on $K$ without the vertex points $p_0,p_1, p_2$ of its convex hull, $h_1$ having boundary values one at vertex $p_1$ and zero at $p_0$ and $p_2$, and $h_2$ having boundary values one at vertex $p_2$ and zero at the other two. Let $\nu_{h_i}$ denote the energy measure of $h_i$, \cite{Kusuoka89, FOT94}. The finite measure $\nu:=\nu_{h_1}+\nu_{h_2}$ is called \emph{Kusuoka's measure}. The bilinear form $(\mathcal{E},\mathcal{F})$ is a strongly local regular Dirichlet form on $L^2(K,\nu)$ admitting a carr\'e du champ $\Gamma$, and the map $h=(h_1,h_2)$ from $K$ onto $h(K)\subset \mathbb{R}^2$ is a homeomorphism, \cite{Ki90, T08}. It is well known that in this situation $\nu$-a.a. spaces $T_x^\ast K$ have dimension one, \cite{Kusuoka89}. The space $\mathcal{Z}$ of cylindrical functions of the form $f=F\circ h$ with $F\in C_b^1(\mathbb{R}^2)$ is an algebra, dense in $\mathcal{F}$, \cite{Ki90, T08}, and, by Stone-Weierstrass, uniformly dense in $C(K)$. Since $\Gamma(h_i)\leq 1$, $i=1,2$, the chain rule (\ref{E:chainrule}) gives $\Gamma(f)\in L^\infty(K,\nu)$ for all elements $f$ of $\mathcal{Z}$. The preceding statements on the density of  $\mathcal{Z}$ remain true if $C_b^1(\mathbb{R}^2)$ is replaced by $C_b^k(\mathbb{R}^k)$, $k=2,...,+\infty$.

We now reset notation and, similarly as in \cite{Strichartz05}, take two identical copies $K_j$, $(\mathcal{E}^{(j)},\mathcal{F}_j)$, $\nu_j$,  $\Gamma^{(j)}$, $\mathcal{Z}_j$, $j=1,2$, of these objects and consider products: We endow $K_1\times K_2$ with the product metric and the product measure $\nu_1\otimes \nu_2$. We define
\begin{equation}\label{E:productform}
\mathcal{E}(f):=\int_{K_1}\mathcal{E}^{(2)}(f(x_1,\cdot))\nu_1(dx_1)+\int_{K_2}\mathcal{E}^{(1)}(f(\cdot,x_2))\nu_2(dx_2)
\end{equation} 
for all $f$ from the space $\mathcal{D}(\mathcal{E})$ of all $f\in L^2(K_1\times K_2,\nu_1\otimes \nu_2)$ such that for $\nu_1$-a.e. $x_1\in K_1$  we have $f(x_1,\cdot)\in \mathcal{F}_2$ and for $\nu_2$-a.e. $x_2\in K_2$  we have $f(\cdot,x_2)\in \mathcal{F}_1$, and for which right-hand side in (\ref{E:productform}) is finite, cf. \cite[Chapter V, Definition 2.1.1]{BH91}. The product Dirichlet form $(\mathcal{E},\mathcal{D}(\mathcal{E}))$ has a carr\'e du champ $\Gamma$, it satisfies 
\begin{equation}\label{E:productcarre}
\Gamma(f_1\otimes f_2)(x_1,x_2)=f_1(x_1)^2\Gamma^{(2)}(f_2)(x_2)+f_2(x_2)^2\Gamma^{(1)}(f_1)(x_1)
\end{equation}
for all $f_i\in \mathcal{F}_i$, $i=1,2$. The space $\mathcal{C}:=\mathcal{Z}_1\otimes \mathcal{Z}_2$ of all finite linear combinations $\sum_k f_{1,k}\otimes f_{2,k}$ with $f_{j,k}\in \mathcal{Z}_j$, interpreted in the usual sense that $\big(\sum_k f_{1,k}\otimes f_{2,k}\big)(x_1,x_2)=\sum_k f_{1,k}(x_1)f_{2,k}(x_2)$, $(x_1,x_2)\in K_1\times K_2$, is an algebra under pointwise multiplication. Again Stone-Weierstrass can be used to see that $\mathcal{C}$ is uniformly dense in $C(K_1\times K_2)$, its density in $\mathcal{D}(\mathcal{E})$ is straightforward from (\ref{E:productform}). So the form $(\mathcal{E},\mathcal{D}(\mathcal{E}))$ is regular. Its strong locality is easily seen, too. For $f_j\in \mathcal{Z}_j$, $j=1,2$, we have $\Gamma(f_1\otimes f_2)\in L^\infty(K_1\times K_2,\nu_1\otimes \nu_2)$, and bilinear extension preserves this property. Consequently Assumptions \ref{A:Dirichletform} and \ref{A:algebraC} are satisfied for $K_1\times K_2$, $\nu_1\otimes \nu_2$, $(\mathcal{E},\mathcal{D}(\mathcal{E}))$ and $\mathcal{C}$. It is quickly seen that at $\nu_1\otimes \nu_2$-a.e. $(x_1,x_2)\in K_1\times K_2$ the space $T_{(x_1,x_2)}(K_1\times K_2)$ is two-dimensional, consequently $\Omega_2(\mathcal{C})$ is nontrivial.
\end{examples}

The idea mentioned above may be used in more general situations. We briefly recall that if $(\mathcal{E},\mathcal{D}(\mathcal{E}))$ is an arbitrary strongly local regular Dirichlet form, then one can choose a suitable core and perform a change of measure to obtain a local Dirichlet form $(\mathcal{E}',\mathcal{D}(\mathcal{E}'))$ and an algebra $\mathcal{C}$ for which Assumptions \ref{A:Dirichletform} and \ref{A:algebraC} are satisfied. This means that, \enquote{up to a change of measure}, the preceding construction of local complexes \enquote{works for any strongly local Dirichlet form}.

\begin{theorem}\label{T:changemeasure}
Suppose that $(X,\varrho)$ is a locally compact separable metric space, $\mu$ a nonnegative Radon measure on $X$ with full support and $(\mathcal{E},\mathcal{D}(\mathcal{E}))$ is a strongly local regular Dirichlet form on $L^2(X,\mu)$. Then there are 
an algebra $\mathcal{C}\subset C_c(X)\cap\mathcal{D}(\mathcal{E})$, dense in $\mathcal{D}(\mathcal{E})$ and in $C_c(X)$, and a finite Borel measure $\mu'$ on $X$ such that $(\mathcal{E},\mathcal{C})$ is closable on $L^2(X,\mu')$ and its closure is a local Dirichlet form $(\mathcal{E}',\mathcal{D}(\mathcal{E}'))$ on $L^2(X,\mu')$ satisfying Assumptions \ref{A:Dirichletform} and \ref{A:algebraC} and $\mathcal{E}'|_\mathcal{C}=\mathcal{E}|_\mathcal{C}$.
\end{theorem}

Theorem \ref{T:changemeasure} is not new at all, see for instance \cite{FukushimaLeJan91, FOT94, FukushimaSatoTaniguchi91, KuwaeNakao91} for variants and related statements. For  the convenience of the reader we sketch a proof in Appendix \ref{App:change}.

\section{Non-local-to-local approximations}\label{S:approx}

Under suitable circumstances the quantities $x\mapsto \left\langle \omega_x,\eta_x\right\rangle_{\Lambda^p T_x^\ast X}$ with $\omega,\eta\in \Omega^p(\mathcal{C})$ appear as limits of analogous quantities in terms of elementary $p$-functions. 

\subsection{Assumptions}

The basic assumption in this section is the following, it implies Assumption \ref{A:Dirichletform}.

\begin{assumption}\label{A:stronglylocal}
We assume that $(X,\varrho)$ is a locally compact separable metric space, $\mu$ a nonnegative Radon measure on $X$ with full support and $(\mathcal{E},\mathcal{D}(\mathcal{E}))$ a strongly local regular Dirichlet form on $L^2(X,\mu)$ with carr\'e du champ $\Gamma$.
\end{assumption}

Given $p\geq 1$, we call $j(x,dy)$ a \emph{generalized kernel from $X$ to $\mathcal{B}(X)^{\otimes p}$} if for any Borel subset $A\in \mathcal{B}(X)^{\otimes p}$ of $X^p$ the map $x\mapsto j(x,A)$ defines an element of $L^0(X,\mu)$ with values in $[0,+\infty]$ and for any Borel subset $B\in\mathcal{B}(X)$ the map $A\mapsto \int_Bj(x,A)\mu(dy)$ is a Borel measure on $X^p$. 

We impose a second assumption regarding the choice of suitable generalized kernels and algebras.

\begin{assumption}\label{A:kernels} There are a directed set $\Theta$ and generalized kernels $j_\theta(x,dy)$, $\theta\in \Theta$, from $X$ to $\mathcal{B}(X)$, and a subalgebra  $\mathcal{C}$ of $\mathcal{D}(\mathcal{E})\cap C_b(X)$ such that 
\begin{enumerate}
\item[(i)] For each $\theta\in \Theta$ the kernel $j_\theta(x,dy)$ is $\mu$-symmetric in the sense that 
\begin{equation}\label{E:symmetry}
j_\theta(x,dy)\mu(dx)=j_\theta(y,dx)\mu(dy),
\end{equation}
\item[(ii)] the algebra $\mathcal{C}$ is dense in dense in $\mathcal{D}(\mathcal{E})$ and dense in $L^1(X,\mu)$,
\item[(iii)] for any $f\in \mathcal{C}$ the functions 
\begin{equation}\label{E:Gammatheta}
x\mapsto \Gamma_\theta(f)(x):=\int_X(f(x)-f(y))^2j_\theta(x,dy),\quad \theta\in \Theta,
\end{equation}
are in $L^1(X,\mu)$ and satisfy
\begin{equation}\label{E:Gammathetalim}
\lim_{\theta\in \Theta}\int_X\Gamma_\theta(f)d\mu=\mathcal{E}(f),
\end{equation}
\item[(iv)] for any sufficiently small $\varepsilon>0$  we have 
\begin{equation}\label{E:killocal}
\lim_{\theta\in \Theta}\int_X\int_{B(x,\varepsilon)^c}(f(x)-f(y))^2j_{\theta}(x,dy)\mu(dx)=0, \quad f\in \mathcal{C}.
\end{equation}
\end{enumerate}
\end{assumption}

The limit relations (\ref{E:Gammathetalim}) and (\ref{E:killocal}) are interpreted in the sense of convergent nets, \cite[Chapter 2]{Kelley75}. This formulation plays no particular role and is used only to accommodate the different notations for corresponding limits in specific realizations below. Likewise, it is only for the sake of easier referencing in specific realizations that we formulate (\ref{E:Gammatheta}), (\ref{E:Gammathetalim}) and (\ref{E:killocal}) primarily for $\mathcal{C}$. 

If (\ref{E:Gammathetalim}) holds for sufficiently many functions, then (\ref{E:killocal}) is immediate from strong locality. 

\begin{lemma}\label{L:local}
Let Assumption \ref{A:stronglylocal} and Assumption \ref{A:kernels} (i), (ii), (iii) be satisfied. Assume in addition that for all  $f\in \mathcal{D}(\mathcal{E})\cap C_c(X)$ the function $\Gamma_\theta(f)$, defined as in (\ref{E:Gammatheta}), is in $L^1(X,\mu)$ and satisfies (\ref{E:Gammathetalim}). Then also Assumption \ref{A:kernels} (iv) holds.
\end{lemma}

Given a set $A\subset X$, we write $A_\varepsilon:=\{x\in X: \dist(x,A)<\varepsilon\}$ for its $\varepsilon$-parallel set.

\begin{proof} 
We first claim that for any nonnegative $\varphi\in \mathcal{D}(\mathcal{E})\cap C_c(X)$ we have 
\begin{equation}\label{E:claimzero}
\lim_{\theta\in \Theta}\int_X\varphi(x)j_\theta(x,B(x,\varepsilon)^c)\mu(dx)=0.
\end{equation}

If the complement of $(\supp\varphi)_\varepsilon$ is empty, then this is clear. Assume that it is not empty and choose $\chi\in \mathcal{D}(\mathcal{E})\cap C_c(X)$ such that $0\leq \chi\leq 1$, $\chi\equiv 1$ on a small open neighborhood of $\supp\varphi$ and $\supp \chi\subset (\supp\varphi)_\varepsilon$. It follows that
\begin{align}
\limsup_{\theta\in \Theta}\int_X\varphi(x)j_\theta(x,B(x,\varepsilon)^c)\mu(dx)&= \limsup_{\theta\in \Theta}\int_X\varphi(x)\int_{B(x,\varepsilon)^c}(\chi(x)-\chi(y))^2j_\theta(x,dy)\mu(dx)\notag\\
&\leq \lim_{\theta\in \Theta}\int_X\varphi(x)\int_X(\chi(x)-\chi(y))^2j_\theta(x,dy)\mu(dx)\notag\\
&=\lim_{\theta\in \Theta}\int_X \varphi\Gamma_\theta(\chi)\:d\mu\notag\\
&=\int_X\varphi \Gamma(\chi)\:d\mu.\notag
\end{align}
The last equality can be seen from (\ref{E:carredef}) and (\ref{E:Gammathetalim}), see for instance the proof of \cite[(3.5) Lemma]{CKS87} for the small calculation. By the strong locality of $\mathcal{E}$ the last line is zero, see \cite[1.5.2. Propri\'et\'es]{LeJan78} or \cite[Corollary 3.2.1]{FOT94}. This shows (\ref{E:claimzero}).

Since for any $f\in\mathcal{C}$ and any $\varphi \in \mathcal{D}(\mathcal{E})\cap C_c(X)$ we have 
\[\big|\int_X\varphi(x) \int_{B(x,\varepsilon)^c}(f(x)-f(y))^2j_\theta(x,dy)\mu(dx)\big|\leq 4\|f\|_{\sup}^2\int_X |\varphi(x)|j_\theta(x,B(x,\varepsilon)^c)\mu(dx),\]
the linear functional 
\[\varphi\mapsto \lim_{\theta\in \Theta}\int_X\varphi(x) \int_{B(x,\varepsilon)^c}(f(x)-f(y))^2j_\theta(x,dy)\mu(dx)\]
is well-defined on $\mathcal{D}(\mathcal{E})\cap C_c(X)$ and seen to be the zero functional. Consequently also its unique extension to all of $C_c(X)$ is zero, and the result follows from the Riesz representation theorem for measures.
\end{proof}

Given $\varepsilon>0$ and $f$ such that $\Gamma_\theta(f)$ is defined, let 
\[\Gamma_\theta^{(\varepsilon)}(f)(x):=\int_{B(x,\varepsilon)}(f(x)-f(y))^2j_\theta(x,dy),\quad \theta\in \Theta.\]

\begin{remark}\label{R:L1} 
Let Assumptions \ref{A:stronglylocal} and \ref{A:kernels} be satisfied and let $\varepsilon>0$.

For any $f\in \mathcal{D}(\mathcal{E})\cap C_b(X)$ such that 
\begin{equation}\label{E:GammathetalimL1}
\lim_{\theta\in \Theta}\Gamma_{\theta}(f)=\Gamma(f)\quad \text{in $L^1(X,\mu)$}
\end{equation}
we have $\lim_{\theta\in \Theta}\Gamma_{\theta}^{(\varepsilon)}(f)=\Gamma(f)$ in $L^1(X,\mu)$. This is clear from (\ref{E:killocal}).
\end{remark}

A third, pragmatic assumption ensures the \enquote{boundedness of gradients}. Together with Assumptions \ref{A:stronglylocal} and \ref{A:kernels} it implies Assumption \ref{A:algebraC}.

\begin{assumption}\label{A:C} 
For any $f\in \mathcal{C}$ we have $\Gamma(f)\in L^\infty(X,\mu)$ and 
\begin{equation}\label{E:Gammathetainfty}
\sup_{\theta\in \Theta}\left\|\Gamma_\theta(f)\right\|_{L^\infty(X,\mu)}<+\infty.
\end{equation}
\end{assumption}

\begin{corollary}\label{C:Gammaepsconv}
Let Assumptions \ref{A:stronglylocal}, \ref{A:kernels} and \ref{A:C} be satisfied. If $\varepsilon>0$ is sufficiently small, then for any $f\in \mathcal{C}$ we have $\lim_{\theta\in \Theta}\Gamma_{\theta}^{(\varepsilon)}(f)=\Gamma(f)$ weakly$^\star$ in $L^\infty(X,\mu)$, that is,
\begin{equation}\label{E:weaklystar}
\lim_{\theta\in \Theta}\int_X g\Gamma_{\theta}^{(\varepsilon)}(f)\:d\mu=\int_X g\Gamma(f)\:d\mu, 
\end{equation}
for any $g\in L^1(X,\mu)$.
\end{corollary}

\begin{proof}
Similarly as before (\ref{E:carredef}),  (\ref{E:Gammathetalim}), a small calculation and (\ref{E:killocal}) show that for any $f,g\in \mathcal{C}$ we have 
(\ref{E:weaklystar}). Since 
\[\big|\int_X g\Gamma(f)\:d\mu\big|\leq \|g\|_{L^1(X,\mu)}\|\Gamma(f)\|_{L^\infty(X,\mu)}\quad\text{and}\quad \big|\int_X g\Gamma_\theta^{(\varepsilon)}(f)\:d\mu\big|\leq \|g\|_{L^1(X,\mu)}\sup_{\theta\in \Theta}\|\Gamma_\theta(f)\|_{L^\infty(X,\mu)}\] 
this extends to arbitrary  $g\in L^1(X,\mu)$.
\end{proof}

\subsection{Pointwise convergence}

Let $D_0:=\emptyset$ and 
\[D_p:=\{(x_0,...,x_p)\in X^{p+1}:\ x_i=x_j\ \text{for some $i\neq j$}\},\quad p\geq 1.\]
We complement (\ref{E:boundedrange}) by defining, for any given $x_0\in X$, 
\[N_{p,x_0}:=\{(x_1,...,x_p)\in X^p:\ (x_0,x_1,...,x_p)\in N_p\}\]
and
\[D_{p,x_0}:=\{(x_1,...,x_p)\in X^p:\ (x_0,x_1,...,x_p)\in D_p\}.\]

For fixed $p\geq 1$ and given $\theta_1, ...,\theta_p\in \Theta$, we use the notation 
\[\overline{\theta}=(\theta_1,...,\theta_p).\] 
For any choice of $A_1,...,A_p\in \mathcal{B}(X)$ the product 
\begin{equation}\label{E:genkerneljp}
j_{p,\overline{\theta}}(x_0,A_1\times \cdots \times A_p):=j_{\theta_1}(x_0,A_1)\cdots j_{\theta_p}(x_0,A_p)
\end{equation}
is an element of $L^1(X,\mu(dx_0))$; we use the notation $\mu(dx_0)$ instead of $\mu$ to indicate the integration variable. This defines a generalized kernel $j_{p,\overline{\theta}}(x_0,d(x_1,...,x_p))$ from $X$ to $\mathcal{B}(X)^{\otimes p}$. Note that for $p=1$ we have $j_{1,\overline{\theta}}=j_\theta$. 

We now, imprecisely speaking, consider the spaces $L^2(N_{p,x_0}\setminus D_{p,x_0},j_{p,\overline{\theta}}(x_0,\cdot))$ and note that for any $f_1,...,f_p\in \mathcal{C}$ and $g\in C_b(X)$ the function 
$(x_1,...,x_p)\mapsto \overline{g}\delta_{p-1}\Alt_p(f_1\otimes \cdots\otimes f_p)(x_0,x_1,...,x_p)$
is an element of $L^2(N_{p,x_0}\setminus D_{p,x_0},j_{p,\overline{\theta}}(x_0,\cdot))$. This is made precise by the following
integrated statement, which is immediate from \cite[Lemma 4.1 (i)]{HinzKommer}.

\begin{lemma} Let Assumptions \ref{A:stronglylocal}, \ref{A:kernels} and \ref{A:C} be satisfied, let $N_\ast=(N_p)_{p\geq 0}$ be a system of diagonal neighborhoods, let $p\geq 1$ and $f_1,...,f_p\in \mathcal{C}$, $g\in C_b(X)$.  If $\varphi\in L^1(X,\mu)$, then 
\begin{multline}
\int_X|\varphi(x_0)|\left\|(\overline{g}\delta_{p-1}\Alt_p(f_1\otimes \cdots\otimes f_p))(x_0,\cdot)\right\|_{L^2(N_{p,x_0}\setminus D_{p,x_0},j_{p,\overline{\theta}}(x_0,\cdot))}^2\mu(dx_0)\notag\\
\leq \|g\|_{\sup}^2\prod_{i=1}^p\|\Gamma_\theta(f_i)\|_{L^\infty(X,\mu)}^2\|\varphi\|_{L^1(X,\mu)}.
\end{multline}
If instead $\varphi\in L^\infty(X,\mu)$, then for any $i=1,...,p$ we have 
\begin{multline}
\int_X|\varphi(x_0)|\left\|(\overline{g}\delta_{p-1}\Alt_p(f_1\otimes \cdots\otimes f_p))(x_0,\cdot)\right\|_{L^2(N_{p,x_0}\setminus D_{p,x_0},j_{p,\overline{\theta}}(x_0,\cdot))}^2\mu(dx_0)\notag\\
\leq \|g\|_{\sup}^2 \|\Gamma_\theta(f_i)\|_{L^1(X,\mu)}^2\prod_{k\neq i}\|\Gamma_\theta(f_k)\|_{L^\infty(X,\mu)}^2\|\varphi\|_{L^\infty(X,\mu)}.
\end{multline}
\end{lemma}

Given a function $\overline{\theta}\mapsto \Phi(\overline{\theta})$, we write 
\begin{equation}\label{E:orderofthings}
\lim_{\overline{\theta}\in \Theta^p}\Phi(\overline{\theta}):=\lim_{\theta_{\sigma(1)}\in \Theta}\cdots \lim_{\theta_{\sigma(p)}\in \Theta}\Phi(\overline{\theta})
\end{equation}
if the limit on the right-hand side exists and does not depend on the chosen order $\sigma\in \mathcal{S}_p$ of individual limits.

We observe the following limit relation, which in particular gives a non-local-to-local convergence of \enquote{cotangential structures} for fixed elements of $\mathcal{C}^p(N_p)$.
\begin{proposition}\label{P:limitx0} Let Assumptions \ref{A:stronglylocal}, \ref{A:kernels} and \ref{A:C} be satisfied, let $N_\ast=(N_p)_{p\geq 0}$ be a system of diagonal neighborhoods and let $p\geq 1$. For any $f_1,...,f_p,h_1,...,h_p\in \mathcal{C}$ and $g,k\in C_b(X)$ we have 
\begin{multline}\label{E:diagonalize}
\lim_{\overline{\theta}\in \Theta^p} \left\langle \overline{g}\delta_{p-1}\Alt_p(f_1\otimes \cdots\otimes f_p),\overline{k}\delta_{p-1}\Alt_p(h_1\otimes\cdots\otimes h_p)\right\rangle_{L^2(N_{p,x_0}\setminus D_{p,x_0},j_{p,\overline{\theta}}(x_0,\cdot))}\\
=\left\langle(gd_0f_1\wedge ...\wedge d_0f_p)_{x_0},(kd_0h_1\wedge...\wedge d_0h_p)_{x_0}\right\rangle_{\Lambda^pT_{x_0}^\ast X}
\end{multline}
weakly$^\star$ in $L^\infty(X,\mu(dx_0))$. If (\ref{E:GammathetalimL1}) holds for all $f\in \mathcal{C}$, then the convergence (\ref{E:diagonalize}) is in $L^1(X,\mu)$.
\end{proposition}

\begin{proof}
We prove (\ref{E:diagonalize}) in the weak$^\star$ sense in $L^\infty(X,\mu)$; the claimed $L^1(X,\mu)$-variant follows by straightforward modifications of the proof.

By (\ref{E:Gammaae}) and (\ref{E:extprodspace}) we have
\begin{equation}\label{E:thewholelot}
\left\langle gd_0f_1\wedge ...\wedge d_0f_p,kd_0h_1\wedge...\wedge d_0h_p\right\rangle_{\Lambda^pT_{x_0}^\ast X}
=g(x_0)k(x_0)\det[\big(\Gamma(f_i,h_j)(x_0)\big)_{i,j=1}^p]
\end{equation}
for $\mu$-a.e. $x_0\in X$. Suppose that $\varepsilon>0$ is given. Writing the matrices in terms of their columns, 
\begin{align}\label{E:telescope}
\big(\Gamma(f_i,h_j)\big)_{i,j=1}^p &-\big(\Gamma_{\theta_j}^{(\varepsilon)}(f_i,h_j)\big)_{i,j=1}^p\notag\\
&=\Big(\big(\Gamma(f_i,h_1)-\Gamma_{\theta_1}^{(\varepsilon)}(f_i,h_1)\big)_{i=1}^p,\big(\Gamma(f_i,h_2)\big)_{i=1}^p,...,\big(\Gamma(f_i,h_p)\big)_{i=1}^p\Big)\notag\\
&+\Big(\big(\Gamma_{\theta_1}^{(\varepsilon)}(f_i,h_1)\big)_{i=1}^p, \big(\Gamma(f_i,h_2)-\Gamma_{\theta_2}^{(\varepsilon)}(f_i,h_2)\big)_{i=1}^p, ...,\big(\Gamma(f_i,h_p)\big)_{i=1}^p\Big)\notag\\
&+\dots\notag\\
&+\Big(\big(\Gamma_{\theta_1}^{(\varepsilon)}(f_i,h_1)\big)_{i=1}^p,..., \Gamma_{\theta_{p-1}}^{(\varepsilon)}(f_i,h_{p-1})\big)_{i=1}^p,\big(\Gamma(f_i,h_p)-\Gamma_{\theta_p}^{(\varepsilon)}(f_i,h_p)\big)_{i=1}^p\Big).
\end{align}
By Corollary \ref{C:Gammaepsconv} and polarization we have 
\[\lim_{\theta_j\in \Theta}\Gamma_{\theta_j}^{(\varepsilon)}(f_i,h_j)=\Gamma(f_i,h_j)\quad\text{weakly$^\star$ in $L^\infty(X,\mu)$}\] 
for all $i$ and $j$. Using Leibniz' formula and (\ref{E:Gammathetainfty}), it follows that  
\begin{align}
&\lim_{\theta_1\in \Theta}\int_X\varphi\:\det\left[\Big(\big(\Gamma(f_i,h_1)-\Gamma_{\theta_1}^{(\varepsilon)}(f_i,h_1)\big)_{i=1}^p,\big(\Gamma(f_i,h_2)\big)_{i=1}^p,...,\big(\Gamma(f_i,h_p)\big)_{i=1}^p\Big)\right]\:d\mu\notag\\
&= \sum_{\sigma\in \mathcal{S}_p}\sgn(\sigma)\lim_{\theta_1\in \Theta}\int_X\varphi\:\big(\Gamma(f_{\sigma(1)},h_1)-\Gamma_{\theta_1}^{(\varepsilon)}(f_{\sigma(1)},h_1)\big)\Gamma(f_{\sigma(2)},h_2)\cdots\Gamma(f_{\sigma(p)},h_p)\:d\mu\notag\\
&=0\notag
\end{align}
for any $\varphi\in L^1(X,\mu)$. Analogous estimates show that also the remaining summands on the right-hands side of (\ref{E:telescope}) converge to zero weakly$^\star$ in $L^\infty(X,\mu)$, so that
\[\lim_{\overline{\theta}\in \Theta}\det[(\Gamma_{\theta_j}^{(\varepsilon)}(f_i,h_j))_{i,j=1}^p]=\det[(\Gamma(f_i,h_j))_{i,j=1}^p]\]
weakly$^\star$ in $L^\infty(X,\mu)$.
Therefore, and again by Leibniz' formula, (\ref{E:thewholelot}) equals
\begin{align}
&\lim_{\overline{\theta}\in \Theta}g(x_0)k(x_0)\sum_{\pi\in \mathcal{S}_p}\sgn(\pi)\int_{B(x_0,\varepsilon)}\cdots\int_{B(x_0,\varepsilon)} (f_{\pi(1)}(x_1)-f_{\pi(1)}(x_0))(h_1(x_1)-h_1(x_0))\times\notag\\
&\hspace{50pt}\cdots \times(f_{\pi(p)}(x_p)-f_{\pi(p)}(x_0))(h_p(x_p)-h_p(x_0))j_{\theta_1}(x_0,dx_1)\cdots j_{\theta_p}(x_0,dx_p)\notag\\
&=\lim_{\overline{\theta}\in \Theta}g(x_0)k(x_0)\int_{B(x_0,\varepsilon)}\cdots\int_{B(x_0,\varepsilon)}\det[(f_j(x_i)-f_j(x_0))_{i,j=1}^p]\prod_{i=1}^p(h_i(x_i)-h_i(x_0))\times\notag\\
&\hspace{250pt}\times j_{\theta_1}(x_0,dx_1)\cdots j_{\theta_p}(x_0,dx_p)\notag\\
&=\lim_{\overline{\theta}\in \Theta}g(x_0)k(x_0)\int_{B(x_0,\varepsilon)}\cdots\int_{B(x_0,\varepsilon)} \Alt_p(\delta_0 f_1(x_0,\cdot)\otimes...\otimes \delta_0f_p(x_0,\cdot))(x_1,...,x_p)\times\label{E:here}\\
&\hspace{70pt}\times \Alt_p(\delta_0 h_1(x_0,\cdot)\otimes...\otimes \delta_0 h_p(x_0,\cdot))(x_1,...,x_p)j_{\theta_1}(x_0,dx_1)\cdots j_{\theta_p}(x_0,dx_p),\notag
\end{align}
seen as weak$^\star$ limits in $L^\infty(X,\mu(dx_0))$. The last equality uses (\ref{E:determinant}) and the fact that $\Alt_p$ is an orthogonal projection in $L^2(B(x_0,\varepsilon)^p\setminus D_{p,x_0},j_{p,\overline{\theta}}(x_0,\cdot))$. By the equality of (\ref{E:deltaalttensor}) and (\ref{E:determinant}) each of the integrals under the limit in (\ref{E:here}) can be rewritten as
\[g(x_0)k(x_0)\int_{B(x_0,\varepsilon)^p}\delta_{p-1}\Alt_p(f_1\otimes...\otimes f_p)(x_0,...,x_p)\delta_{p-1}\Alt_p(h_1\otimes...\otimes h_p)(x_0,...,x_p) j_{\overline{\theta},p}(x_0,d(x_1,...,x_p)).\]
To compare this last line to 
\begin{multline}\label{E:desired}
\int_{B(x_0,\varepsilon)^p}\overline{g}(x_0,...,x_p)\overline{k}(x_0,...,x_p)\delta_{p-1}\Alt_p(f_1\otimes...\otimes f_p)(x_0,...,x_p)\delta_{p-1}\Alt_p(h_1\otimes...\otimes h_p)(x_0,...,x_p)\times \\
\times j_{\overline{\theta},p}(x_0,d(x_1,...,x_p)),
\end{multline}
note that, given arbitrary $\gamma>0$, we can choose $\varepsilon$ so small that 
\[|g(x_0)k(x_0)-\overline{g}(x_0,...,x_p)\overline{k}(x_0,...,x_p)|<\gamma\quad \text{for all $x_1,...,x_p\in B(x_0,\varepsilon)$}.\] 
Integrating the two lines in question against $\varphi\in L^1(X,\mu(dx_0))$, the modulus of the difference of the resulting integrals is bounded by 
\begin{multline}
\gamma\int_X|\varphi(x_0)|\int_{B(x_0,\varepsilon)^p}|\delta_{p-1}\Alt_p(f_1\otimes...\otimes f_p)(x_0,...,x_p)\delta_{p-1}\Alt_p(h_1\otimes...\otimes h_p)(x_0,...,x_p)|\times\notag\\
\qquad\qquad\times j_{\overline{\theta},p}(x_0,d(x_1,...,x_p))\mu(dx_0)
\notag\\
\leq \gamma \left(\int_X|\varphi(x_0)|\int_{B(x_0,\varepsilon)^p}\big(\delta_{p-1}\Alt_p(f_1\otimes...\otimes f_p)(x_0,...,x_p)\big)^2 j_{\overline{\theta},p}(x_0,d(x_1,...,x_p))\mu(dx_0)\right)^{1/2}\times \notag\\
\quad\quad\times \left(\int_X|\varphi(x_0)|\int_{B(x_0,\varepsilon)^p}\big(\delta_{p-1}\Alt_p(h_1\otimes...\otimes h_p)(x_0,...,x_p)\big)^2 j_{\overline{\theta},p}(x_0,d(x_1,...,x_p))\mu(dx_0)\right)^{1/2}.\notag
\end{multline}
By the preceding the limit along $\theta\in \Theta$ of this last quantity is 
\[\gamma \left(\int_X|\varphi|\det\big[(\Gamma(f_i,f_j))_{i,j=1}^p\big]d\mu\right)^{1/2}\left(\int_X|\varphi|\det\big[(\Gamma(h_i,h_j))_{i,j=1}^p\big]d\mu\right)^{1/2},\]
and it can be made arbitrarily small by a suitable choice of $\gamma$. This shows that (\ref{E:here}) equals the limit of (\ref{E:desired}). We may assume that $\varepsilon$ is small enough to have $B(x_0,\varepsilon)^p\subset N_p$, so that another application of Lemma \ref{L:local} gives the equality of (\ref{E:here}) and 
\begin{multline}
\lim_{\overline{\theta}\in \Theta}\int_{N_{p,x_0}\setminus D_{p,x_0}}\overline{g}(x_0,...,x_p)\delta_{p-1}\Alt_p(f_1\otimes...\otimes f_p)(x_0,...,x_p)\times\notag\\
\times\overline{k}(x_0,...,x_p)\delta_{p-1}\Alt_p(h_1\otimes...\otimes h_p)(x_0,...,x_p)j_{p,\overline{\theta}}(x_0,d(x_1,...,x_p)),
\end{multline}
seen as weak$^\star$ limits in $L^\infty(X,\mu(dx_0))$. 
\end{proof}

Setting
\begin{equation}\label{E:Jp}
J_{p,\overline{\theta}}(d(x_0,...,x_p)):=\frac{1}{p+1}\sum_{k=0}^pj_{p,\overline{\theta}}(x_k,d(x_0,..., \hat{x}_k, ..., x_p))\mu(dx_k),\quad p\geq 1,
\end{equation}
we obtain Borel measures $J_{p,\overline{\theta}}$ on $X^{p+1}$, symmetric in $x_0,...,x_p$. 

The following connection between the spaces $L^2(N_{p}\setminus D_{p},J_{p,\overline{\theta}})$ and the direct integrals $L^2(X,\Lambda^pT^\ast X,\mu)$ is immediate from Proposition \ref{P:limitx0} and its proof.

\begin{corollary}\label{C:limit}
Let Assumptions \ref{A:stronglylocal}, \ref{A:kernels} and \ref{A:C} be satisfied, let $N_\ast=(N_p)_{p\geq 0}$ be a system of diagonal neighborhoods and let $p\geq 1$. If $\mathcal{C}\subset C_c(X)$, then for any $f_1,...,f_p,h_1,...,h_p\in \mathcal{C}$ and $g,k\in C_b(X)$ we have 
\begin{multline}\label{E:diagonalizesum}
\lim_{\overline{\theta}\in \Theta^p} \left\langle \overline{g}\delta_{p-1}\Alt_p(f_1\otimes \cdots\otimes f_p), \overline{k}\delta_{p-1}\Alt_p(h_1\otimes\cdots\otimes h_p)\right\rangle_{L^2(N_{p}\setminus D_{p},J_{p,\overline{\theta}})}\\
=\left\langle gd_0f_1\wedge ...\wedge d_0f_p, kd_0h_1\wedge...\wedge d_0h_p\right\rangle_{L^2(X,\Lambda^pT^\ast X,\mu)}.
\end{multline}
\end{corollary}

\begin{remark}\label{R:quadraticforms}
A special case of Corollary \ref{C:limit} gives
\begin{multline}\label{E:quadraticforms}
\lim_{\overline{\theta}\in \Theta^p} \left\langle \delta_p(\overline{g}\delta_{p-1}\Alt_p(f_1\otimes \cdots\otimes f_p)),\delta_p\overline{k}\delta_{p-1}\Alt_p(h_1\otimes\cdots\otimes h_p)\right\rangle_{L^2(N_{p}\setminus D_{p},J_{p,\overline{\theta}})}\\
=\left\langle d_p(gd_0f_1\wedge ...\wedge d_0f_p), d_p(kd_0h_1\wedge...\wedge d_0h_p)\right\rangle_{L^2(X,\Lambda^pT^\ast X,\mu)}.
\end{multline}
We may then view the linear operator $(\delta_p,\mathcal{C}^p(N_p))$ as densely defined in $L^2(N_{p}\setminus D_{p},J_{p,\overline{\theta}})$ and mapping into $L^2(N_{p+1}\setminus D_{p+1},J_{p+1,\overline{\theta}})$, \cite[Proposition 5.1]{HinzKommer}. If $\delta_p^\ast$ denotes its adjoint, then $\delta_p^\ast\delta_p$ is a non-local variant of the lower Hodge Laplacian $d_p^\ast d_p$. In this context the bilinear extension of the limit relation (\ref{E:quadraticforms}) provides the non-local-to-local convergence of quadratic forms associated with lower Hodge Laplacians \enquote{pointwise} on $\mathcal{C}^p(N_p)$.
\end{remark}

\subsection{Localization maps}

Under the assumptions of the last subsection we can use formula (\ref{E:niceformula}) and linearity to define localization maps which link the complexes in (\ref{E:elemcomplex}) and (\ref{E:localcomplex}). The correctness of a corresponding definition is now ensured by Corollary \ref{C:limit}. 

Given $p\geq 1$ and a function $F\in \mathcal{C}^p(N_p)$
with representation 
\begin{equation}\label{E:repF}
F=\sum_i \overline{g}^{(i)}\delta_{p-1}\Alt_p(f_1^{(i)}\otimes ... \otimes f_p^{(i)}),
\end{equation}
where $g\in \mathcal{C}\oplus \mathbb{R}$ and $f_1,...,f_p\in \mathcal{C}$, we write 
\begin{equation}\label{E:lambdapgeneral}
\lambda_p(F):=\sum_i g^{(i)}d_0f_1\wedge ... \wedge d_0f_p.
\end{equation}
Given $f\in \mathcal{C}^0(N_0)=\mathcal{C}$, let 
\begin{equation}\label{E:lambda0general}
\lambda_0(f):=f. 
\end{equation}
The following result generalizes Theorem \ref{T:localizationM} to a wide variety of metric measure spaces endowed with a strongly local regular Dirichlet form. Recall the definition (\ref{E:OmegapC}) of the spaces $\Omega^p(\mathcal{C})$.

\begin{theorem}\label{T:localization} Let Assumptions \ref{A:stronglylocal}, \ref{A:kernels} and \ref{A:C} be satisfied and let $N_\ast=(N_p)_{p\geq 0}$ be a system of diagonal neighborhoods.
\begin{enumerate}
\item[(i)] For any integer $p\geq 0$ the assignment (\ref{E:lambdapgeneral}) respectively (\ref{E:lambda0general})
defines a linear surjection 
\begin{equation}\label{E:nicemap}
\lambda_p:\mathcal{C}^p(N_p)\to \Omega^p(\mathcal{C}).
\end{equation}
\item[(ii)] The family $\lambda_\ast=(\lambda_p)_{p\geq 0}$ defines a cochain map $\lambda_\ast:(\mathcal{C}^\ast(N_\ast),\delta_\ast)\to (\Omega^\ast(\mathcal{C}),d_\ast)$,
that is,
\[d_p\circ \lambda_p=\lambda_{p+1}\circ \delta_p,\quad p\geq 0.\]
In particular, it induces well-defined linear maps 
\[\lambda_p^\ast:H^p \mathcal{C}^\ast(N_\ast)\to H^p\Omega^\ast(\mathcal{C}),\quad p\geq 0,\] 
between the cohomologies (\ref{E:elemcoho}) and (\ref{E:localcoho}).
\end{enumerate}
\end{theorem}

\begin{proof}
To see (i), note that for any $p\geq 1$ the definition of $\lambda_p$ using (\ref{E:lambdapgeneral}) does not depend on the particular representation (\ref{E:repF})  of the given elementary $p$-function $F$: By the linearity in (\ref{E:repF}) and (\ref{E:lambdapgeneral}) it suffices to note that if  
\[\sum_i \overline{g}^{(i)}\delta_{p-1}\Alt_p(f_1^{(i)}\otimes ... \otimes f_p^{(i)})=0,\] 
then (\ref{E:diagonalize}) implies that
\[\sum_i g^{(i)}d_0f_1\wedge ... \wedge d_0f_p=0\quad \text{in $L^2(X,\Lambda^pT^\ast X,\mu)$}.\]
The linearity of (\ref{E:nicemap}) is obvious, its surjectivity is immediate from (\ref{E:OmegapC}). Statement (ii) follows as in the proof of Theorem \ref{T:localizationM}.
\end{proof}

\section{Several realizations}\label{S:realize}

We discuss several choices of kernels $j_\theta(x,dy)$ and algebras $\mathcal{C}$ resulting in concrete realizations of Assumptions \ref{A:stronglylocal}, \ref{A:kernels} and \ref{A:C}.

\subsection{Local averages on manifolds}

On Riemannian manifolds one can use a variant of (\ref{E:diagonalize}) based on geometric averages in the style of \cite{BuragoIvanovKurylev14}. 

Let $M$ be a smooth Riemannian manifold of dimension $n$, let $\mu$ denote the Riemannian volume on $M$ and $(\mathcal{E},W^1_0(M))$ the Dirichlet integral as defined in (\ref{E:Dirichletintegral}). Choose $\mathcal{C}=C_c^\infty(M)$. We write $\nu_n$ for the volume of the unit ball in $\mathbb{R}^n$. Let $r_0>0$ and consider the generalized kernels 
\[j_r(x,dy)=\frac{n+2}{\nu_n r^{n+2}}\mathbf{1}_{B(x,r)}(y)\mu(dy),\quad 0<r<r_0;\]
in this case $\Theta=(0,r_0)$. We write 
\[\Gamma_r(f)(x)=\frac{n+2}{\nu_n r^{n+2}}\int_{B(x,r)}(f(x)-f(y))^2\mu(dy),\quad 0<r<r_0.\]

\begin{corollary}\label{C:averagesrmf}
Given $M$, $(\mathcal{E},W^1_0(M))$, $j_r(x,dy)$, $0<r<r_0$, and $\mathcal{C}$ as stated, Assumptions \ref{A:stronglylocal}, \ref{A:kernels} and \ref{A:C} are satisfied. Moreover, $\lim_{r\to 0}\Gamma_r(f)=\Gamma(f)$ pointwise at all $x\in M$ and in $L^1(M)$ for all $f\in \mathcal{C}$. 

The limit relation (\ref{E:diagonalize}) holds in $L^1(M)$ with $\lim_{\overline{\theta}\in \Theta}$, $X$ and $j_{p,\overline{\theta}}(x_0,\cdot)$ replaced by $\lim_{\overline{r}\to 0}$, $M$ and $j_{p,\overline{r}}(x_0,\cdot)$, it also holds pointwise at all $x\in M$.
\end{corollary}

\begin{proof} The symmetry condition (\ref{E:symmetry}) obviously holds for $j_r(x,dy)$ and $\mu$, and clearly the algebra $C_c^\infty(M)$ is dense in $W^1_0(M)$ and in $L^1(M)$. Since for any $f\in C_c^\infty(M)$ we have 
\begin{equation}\label{E:uniL1bound}
\Gamma_r(f)\leq  \frac{n+2}{\nu_n} \lip(f)^2\mathbf{1}_{(\supp f)_{r_0}},
\end{equation}
each $\Gamma_r(f)$ is in $L^1(M)$. Using the exponential map and Taylor expansion, it was shown in \cite[Section 2.3]{BuragoIvanovKurylev14} that for any $f\in C_c^\infty(M)$ we have 
\[\lim_{r\to 0} \frac{2(n+2)}{\nu_n r^{n+2}}\int_{B(x,r)}(f(y)-f(x))\mu(dy)=\Delta f(x),\quad x\in M,\]
where $\Delta$ denotes the Laplace-Beltrami operator. Since 
\[\Gamma(f)=\frac12\Delta f^2-f\Delta f,\]
the pointwise convergence of $\Gamma_r(f)$ to $\Gamma(f)$ follows, and taking into account (\ref{E:uniL1bound}) also the convergence in $L^1(M)$. Clearly $\Gamma(f)=|\nabla f|^2$ is bounded for $f\in C_c^\infty(M)$, and (\ref{E:uniL1bound})  gives (\ref{E:Gammathetainfty}). Condition (\ref{E:killocal}) is obvious. The claimed pointwise convergence can be seen following the proof of Proposition \ref{P:limitx0} with straightforward modifications.
\end{proof}

\subsection{Semigroup approximation}

A versatile variant of (\ref{E:diagonalize}) can be formulated using the well-known semigroup approximation for $\mathcal{E}$, see \cite[Chapter I, Proposition 3.3.1]{BH91} or \cite[Lemma 1.3.4. (i)]{FOT94}. 

Suppose that Assumption \ref{A:stronglylocal} is satisfied. Let $(P_t)_{t>0}$ be the symmetric Markov semigroup on $L^2(X,\mu)$ uniquely associated with $(\mathcal{E},\mathcal{D}(\mathcal{E}))$, \cite{BH91, FOT94}. Writing 
\[P_t(x,A):=P_t\mathbf{1}_A(x)\] 
for each fixed $t>0$, we find that the $L^0(X,\mu)$-functions
\[x\mapsto j_{t}(x,A):=\frac{1}{2t}P_t(x,A),\quad A\in \mathcal{B}(X),\] 
define generalized kernels $j_t(x,dy)$, $t>0$, from $X$ to $\mathcal{B}(X)$. Given $p\geq 1$, we consider the generalized kernels 
\begin{equation}\label{E:kernelsp}
j_{p,\overline{t}}(x_0,d(x_1,...,x_p))=j_{t_1}(x_0,dx_1)\cdots j_{t_p}(x_0,dx_p)=\frac{1}{2^pt_1\cdots t_p}P_t(x_0,dx_1)\cdots P_t(x_0,dx_p)
\end{equation}
from $X$ to $\mathcal{B}(X)^{\otimes p}$ with index $\overline{t}=(t_1,...,t_p)\in (0,+\infty)^p$.

Given $f\in \mathcal{D}(\mathcal{E})$ and $t>0$, we set 
\begin{equation}\label{E:Gammat}
\Gamma_t(f)(x):=\int_X(f(x)-f(y))^2j_t(x,dy),\quad t>0.
\end{equation}
Recall that $(P_t)_{t>0}$ is said to be \emph{conservative} if $P_t\mathbf{1}=\mathbf{1}$ $\mu$-a.e. for all $t>0$.

\begin{corollary}\label{C:semigroups}
Let Assumption \ref{A:stronglylocal} be satisfied. If $(P_t)_{t>0}$ is conservative, then Assumption \ref{A:kernels} holds with the kernels $j_t(x,dy)$, $0<t<1$. 

If in addition $\mathcal{C}$ is a subalgebra of $\mathcal{D}(\mathcal{E})\cap C_b(X)$, dense in $\mathcal{D}(\mathcal{E})$ and in $L^1(X,\mu)$, and such that for any $f\in \mathcal{C}$ we have $\Gamma(f)\in L^\infty(X,\mu)$ and 
\begin{equation}\label{E:Gammatinfty}
\sup_{0<t<1}\left\|\Gamma_t(f)\right\|_{L^\infty(X,\mu)}<+\infty,\quad f\in \mathcal{C},
\end{equation}
then also Assumption \ref{A:C} holds. 

If both is true, then (\ref{E:diagonalize}) holds weakly in $L^1(X,\mu)$ with $\lim_{\overline{\theta}\in \Theta}$ and $j_{p,\overline{\theta}}(x_0,\cdot)$ replaced by $\lim_{\overline{t}\to 0}$ and $j_{p,\overline{t}}(x_0,\cdot)$ as in (\ref{E:kernelsp}).
\end{corollary}

\begin{proof}
Condition (\ref{E:symmetry}) is immediate from the symmetry of $(P_t)_{t>0}$. In the conservative case we have
\begin{equation}\label{E:spectral}
\mathcal{E}(f)=\sup_{t>0}\left\|\Gamma_t(f)\right\|_{L^1(X,\mu)}
\end{equation} 
for all $f\in \mathcal{D}(\mathcal{E})$ by symmetry and the mentioned semigroup approximation for $\mathcal{E}$.
\end{proof}

Suppose that Assumption \ref{A:stronglylocal} holds. Let $(\mathcal{L},\mathcal{D}(\mathcal{L}))$ denote the infinitesimal generator of $(\mathcal{E},\mathcal{D}(\mathcal{E}))$, that is, the unique non-positive definite self-adjoint operator on $L^2(X,\mu)$ such that 
\[\mathcal{E}(f,g)=-\left\langle \mathcal{L}f,g\right\rangle_{L^2(X,\mu)},\quad f\in \mathcal{D}(\mathcal{L}),\quad g\in \mathcal{D}(\mathcal{E}).\]
Let $(\mathcal{L}^{(1)},\mathcal{D}(\mathcal{L}^{(1)}))$ denote the smallest closed extension in $L^1(X,\mu)$ of the restriction of $\mathcal{L}$ to 
\[\{f\in \mathcal{D}(\mathcal{L})\cap L^1(X,\mu): \mathcal{L}f\in L^1(X,\mu)\},\] 
that is, the infinitesimal generator of the strongly continuous semigroup on $L^1(X,\mu)$, obtained through the unique continuation of the restricted operators $P_t|_{L^2(X,\mu)\cap L^1(X,\mu)}$. The space $\mathcal{D}(\mathcal{L}^{(1)})\cap L^\infty(X,\mu)$ is an algebra and a dense subspace of $\mathcal{D}(\mathcal{E})$, \cite[Chapter I, Theorem 4.2.1 and Lemma 4.2.1.1]{BH91}. If $(P_t)_{t>0}$ is conservative, then the limit
\begin{equation}\label{E:Gammalim}
\Gamma(f)(x)=\lim_{t\to 0}\Gamma_t(f)(x) 
\end{equation}
exists in $L^1(X,\mu)$ for any $f\in \mathcal{D}(\mathcal{L}^{(1)})\cap L^\infty(X,\mu)$; this is a consequence of symmetry and
\cite[Chapter I, Theorem 4.2.1]{BH91}. If in addition $\mathcal{C}$ is a subalgebra of $\mathcal{D}(\mathcal{L}^{(1)})\cap C_b(X)$, dense in $\mathcal{D}(\mathcal{E})$ and such that (\ref{E:Gammatinfty}) holds, then (\ref{E:diagonalize}) holds in $L^1(X,\mu)$ with $\lim_{\overline{\theta}\in \Theta}$ and $j_{p,\overline{\theta}}(x_0,\cdot)$ replaced by $\lim_{\overline{t}\to 0}$ and $j_{p,\overline{t}}(x_0,\cdot)$.

\begin{examples}\label{Ex:Rmf2}
As in Examples \ref{Ex:Rmf}, let $(M,\mu)$ be a weighted manifold of dimension $n$, let $(\mathcal{E},W_0^1(M))$ be the Dirichlet integral (\ref{E:Dirichletintegral}) and  $\mathcal{C}=C_c^\infty(M)$. We assume in addition that $M$ is stochastically complete,
\cite[Section 11.4]{Grigoryan2009}, that is, $(P_t)_{t>0}$ is conservative. Then all hypotheses of Corollary \ref{C:semigroups} are satisfied. Note that for any $f\in C_c^\infty(M)$ we have
\begin{equation}\label{E:Gammatrick1}
\|\Gamma_t(f)\|_{\sup}\leq \frac12\left\|\Delta_\mu f^2\right\|_{\sup}+\|f\|_{\sup}\left\|\Delta_\mu f\right\|_{\sup},\quad t>0,
\end{equation}
where $\Delta_\mu$ denotes the Laplacian on the weighted manifold, \cite[Section 3.6]{Grigoryan2009}. This follows from
\[\Gamma_t(f)=\frac{1}{t}\left\lbrace \frac12(P_tf^2-f^2)-f(P_tf-f)\right\rbrace,\quad t>0,\]
and the estimate 
\begin{equation}\label{E:Gammatrick2}
\Big\|\frac{P_th-h}{t}\Big\|_{\sup}=\Big\|\frac{1}{t}\int_0^t P_s\Delta_\mu h\:ds\Big\|_{\sup}\leq \frac{1}{t}\int_0^t\left\|P_s\Delta_\mu h\right\|_{\sup}ds\leq \left\|\Delta_\mu h\right\|_{\sup},
\end{equation}
valid for any $h\in C_c^\infty(M)$. By the preceding remarks the limit relation (\ref{E:Gammalim}) holds in $L^1(M)$, and also the convergence (\ref{E:diagonalize}) holds in $L^1(M)$.
\end{examples}

\begin{examples}\label{Ex:RCDKN}
Let  $K\in \mathbb{R}$, $1\leq N<+\infty$, suppose that $(X,\varrho,\mu)$ is an $\mathrm{RCD}^\ast(K,N)$-space, \cite{EKS15}, and $(\mathcal{E},\mathcal{D}(\mathcal{E}))$ is the Cheeger energy. Given $\lambda>0$ we write 
\[G_\lambda f=\int_0^\infty e^{-\lambda t}P_tf\:dt,\quad f\in L^2(X,\mu),\]
for the $\lambda$-resolvent of $f$, \cite{BH91,FOT94}. Consider the space 
\[\mathcal{C}:=\big\lbrace f\in C_b(X)\cap\mathcal{D}(\mathcal{L})\cap\mathcal{D}(\mathcal{L}^{(1)}):\ \Gamma(f)\in L^\infty(X,\mu)\ \text{and}\ \mathcal{L}f\in L^\infty(X,\mu)\big\rbrace.\]
By definition $\Gamma(f)\in L^\infty(X,\mu)$, $f\in \mathcal{C}$. By \cite[Proposition 12.4]{HMS23} the space $\mathcal{C}$ is an algebra, and there is some $\lambda_0>0$ such that $\mathcal{C}$ contains $\{G_\lambda f:\ f\in C_c(X),\ \lambda>\lambda_0\}$. Since this set is dense in $\mathcal{D}(\mathcal{E})$ and in $L^1(X,\mu)$, the same is true for $\mathcal{C}$. Since by definition $\mathcal{L}f\in L^\infty(X,\mu)$, $f\in \mathcal{C}$, condition (\ref{E:Gammatinfty}) can be seen using (\ref{E:Gammatrick1}) and (\ref{E:Gammatrick2}) with $\|\cdot\|_{L^\infty(X,\mu)}$ in place of $\|\cdot\|_{\sup}$. By the preceding remarks it follows again that the limits in (\ref{E:Gammalim}) and in (\ref{E:diagonalize}) exist in $L^1(M)$.
\end{examples}

\begin{examples}\label{Ex:degenerate}
Let $X=\mathbb{R}^n$ and let $\mu(dx)=dx$ be the $n$-dimensional Lebesgue measure. Consider the bilinear form 
\begin{equation}\label{E:degenerate}
(f,g)\mapsto \sum_{i,j=1}^n\int_{\mathbb{R}^n}a_{ij}\frac{\partial f}{\partial x_i}\frac{\partial g}{\partial x_j}dx, \quad f,g\in C_c^\infty(\mathbb{R}^n),
\end{equation}
where $a_{ij}=a_{ji}$ are bounded measurable real valued coefficients such that $(a_{ij})_{i,j=1}^n$ is positive definite a.e. 
and for each $i$ and $j$ we have $\frac{\partial}{\partial x_i}a_{ij}\in L^2_{\loc}(\mathbb{R}^n)$. This form is closable on $L^2(\mathbb{R}^n)$, \cite[Chapter II, Section 1]{MaR92}, and its closure $(\mathcal{E},\mathcal{D}(\mathcal{E}))$ is a strongly local regular Dirichlet form. Its infinitesimal generator $(\mathcal{L},\mathcal{D}(\mathcal{L}))$ is the Friedrichs extension of a possibly degenerate operator of second order, and $\mathcal{D}(\mathcal{L})$ contains $\mathcal{C}:=C_c^\infty(\mathbb{R}^n)$. The associated symmetric Markov semigroup $(P_t)_{t>0}$ is conservative, see for instance \cite[Theorem 3.7]{ERSZ07} and \cite[p. 74]{ERSZ06}. Using again variants of (\ref{E:Gammatrick1}) and (\ref{E:Gammatrick2}) we find that all hypotheses of Corollary \ref{C:semigroups} are satisfied.
\end{examples}

\begin{examples}\label{Ex:ProdSG2}
As in Examples \ref{Ex:ProdSG}, let $K\subset \mathbb{R}^2$ be the Sierpinski gasket and $(\mathcal{E},\mathcal{F})$ the standard resistance form on $K$. Let also $\nu$ and $h$ be as there. Now let $\mathcal{Z}$ be the algebra of cylindrical functions $f=F\circ h$ with $F\in C_b^2(\mathbb{R}^2)$. As before we have $\Gamma(f)\in L^\infty(X,\nu)$, $f\in \mathcal{Z}$. Let $(\mathcal{L},\mathcal{D}(\mathcal{L}))$ denote the inifinitesimal generator of the Dirichlet form $(\mathcal{E},\mathcal{F})$ on $L^2(K,\nu)$. By \cite[Corollary 6.1]{T08} we have $\mathcal{L}f\in L^\infty(X,\nu)$ for all $f\in \mathcal{Z}$.

We reset notation and, as in Examples \ref{Ex:ProdSG}, take two identical copies $K_j$, $(\mathcal{E}^{(j)},\mathcal{F}_j)$, $\nu_j$, $\Gamma^{(j)}$, $\mathcal{Z}_j$, $j=1,2$, of these objects and consider products: We endow $K_1\times K_2$ with the measure $\nu_1\otimes \nu_2$ and write $(\mathcal{E},\mathcal{D}(\mathcal{E}))$ for the product Dirichlet form (\ref{E:productform}) with domain $\mathcal{D}(\mathcal{E})$ as described in Examples \ref{Ex:ProdSG} and $\Gamma$ for the carr\'e du champ as in (\ref{E:productcarre}). Let $(P_t^{(j)})_{t>0}$, $j=1,2$, denote the symmetric Markov semigroups on $L^2(K_j,\nu_j)$ uniquely associated with $(\mathcal{E}_j,\mathcal{F}_j)$, respectively. They can be extended to symmetric Markov semigroups on $L^2(K_1\times K_2,\nu_1\otimes \nu_2)$ by setting 
\[P_t^{(1)}f(x_1,x_2):=P_t^{(1)}(f(\cdot,x_2))(x_1)\quad \text{and}\quad P_t^{(2)}f(x_1,x_2):=P_t^{(2)}(f(x_1,\cdot))(x_2),\] 
see \cite[Chapter V, Section 2.1]{BH91}. They commute, and the operators $P_t$ of the symmetric Markov semigroup $(P_t)_{t>0}$ on $L^2(K_1\times K_2,\nu_1\otimes \nu_2)$ uniquely associated with $(\mathcal{E},\mathcal{D}(\mathcal{E}))$ satisfy $P_t=P_t^{(1)}P_t^{(2)}=P_t^{(2)}P_t^{(1)}$, $t>0$, \cite[Chapter V, Proposition 2.1.3]{BH91}. Since the semigroups $(P_t^{(j)})_{t>0}$, $j=1,2$ are conservative, so is $(P_t)_{t>0}$. Given $f_j\in \mathcal{Z}_j$, $j=1,2$, we find that 
\begin{align}
\Gamma_t(f_1\otimes f_2)(x_1,x_2)&=\frac{1}{2t}\int_{K_1\times K_2}(f_1(x_1)f_2(x_2)-f_1(y_1)f_2(y_2))^2P_t^{(1)}(x_1,dy_1)P_t^{(2)}(x_2,dy_2)\notag\\
&\leq f_1(x_1)^2\frac{1}{t}\int_{K_2}(f_2(x_2)-f_2(y_2))^2P_t^{(2)}(x_2,dy_2)\notag\\
&\qquad\qquad P_t^{(2)}(f_2^2)(x_2)\frac{1}{t}\int_{K_1}(f_1(x_1)-f_1(y_1))^2P_t^{(1)}(x_1,dy_1)\notag\\
&\leq 2\|f_1\|^2_{L^\infty(K_1,\nu_1)}\|\Gamma^{(2)}_t(f_2)\|_{L^\infty(K_2,\nu_2)}+2\|f_2\|^2_{L^\infty(K_2,\nu_2)}\|\Gamma^{(1)}_t(f_1)\|_{L^\infty(K_1,\nu_1)}.\notag
\end{align}
This implies (\ref{E:Gammatinfty}). It follows that $K_1\times K_2$, $\nu_1\otimes\nu_2$, $(\mathcal{E},\mathcal{D}(\mathcal{E}))$ and $\mathcal{C}:=\mathcal{Z}_1\otimes \mathcal{Z}_2$ satisfy all hypotheses of Corollary \ref{C:semigroups}.
\end{examples}

\begin{remark}\label{R:subR} The hypotheses of Corollary \ref{C:semigroups} are also satisfied for certain sub-Riemannian geometries. However, in such cases the interpretation of $\Omega^\ast(\mathcal{C})$ may need more discussion.
\end{remark}

\subsection{L\'evy kernel approximation}

A variant of the preceding semigroup approximation is an approximation using \enquote{fractional} kernels in the spirit of \cite{BBM01}. Given $0<\alpha<1$, let 
\[\nu_\alpha(dt)=\frac{\alpha\:dt}{\mathbb{\Gamma}(1-\alpha)t^{\alpha+1}}\]
be the L\'evy measure of the $\alpha$-stable subordinator, \cite{Jacob01, Sato99}; here $\mathbb{\Gamma}$ denotes the Euler Gamma function. 

Let Assumption \ref{A:stronglylocal} be in force. Then the $L^\infty(X,\mu)$-functions 
\[x\mapsto j_\alpha(x,A):=\frac12\int_{0+}^\infty P_t(x,A)\nu_\alpha(dt),\quad A\in \mathcal{B}(X),\]
define generalized kernels $j_\alpha(x,dy)$, $0<\alpha<1$, from $X$ to $\mathcal{B}(X)$. 

\begin{examples}
Suppose that $X=M$ is a $n$-dimensional weighted manifold as in Examples \ref{Ex:Rmf} and that its heat kernel $p_t(x,y)$, \cite{Da89, Grigoryan2009}, admits 
two-sided Gaussian estimates of the form
\[c^{-1}t^{-\frac{n}{2}}\exp\big(-c\frac{\varrho(x,y)^2}{t}\big)\leq p(t,x,y)\leq ct^{-\frac{n}{2}}\exp\big(-\frac{\varrho(x,y)^2}{ct}\big),\quad x,y\in M,\]
for all $t>0$, where $c>1$ is a universal constant and $\varrho$ denotes the geodesic distance. Then $P_t(x,dy)=p_t(x,y)\mu(dy)$, and  $j_\alpha(x,dy)=j_\alpha(x,y)\mu(dy)$ with symmetric density $j_\alpha(x,y)$ satisfying the estimate
\[c^{-1}\varrho(x,y)^{-n-2\alpha}\leq j_\alpha(x,y)\leq c\varrho(x,y)^{-n-2\alpha},\quad x,y\in M, \]
with a universal constant $c>1$. In this case we have 
\[c^{-1}\:\int_M\frac{(f(x)-f(y))^2}{\varrho(x,y)^{n+2\alpha}}\mu(dy)\leq \Gamma_\alpha(f)(x)\leq c\:\int_M\frac{(f(x)-f(y))^2}{\varrho(x,y)^{n+2\alpha}}\mu(dy),\quad x\in M,\]
for any $f\in \mathcal{C}$.
\end{examples}

Given $p\geq 1$, we consider the generalized kernels
\begin{equation}\label{E:fractionalkernelsp}
j_{p,\overline{\alpha}}(x_0,d(x_1,...,x_p))=j_{\alpha_1}(x_0,dx_1)\cdots j_{\alpha_p}(x_0,dx_p)
\end{equation}
from $X$ to $\mathcal{B}(X)^{\otimes p}$ with index $\overline{\alpha}=(\alpha_1,...,\alpha_p)\in (0,1)^p$.


\begin{corollary}\label{C:fractional}
Let Assumption \ref{A:stronglylocal} be satisfied. If $(P_t)_{t>0}$ is conservative, then Assumption \ref{A:kernels} holds with the kernels $j_\alpha(x,dy)$, $0<\alpha<1$. 

If in addition $\mathcal{C}$ is a subalgebra of $\mathcal{D}(\mathcal{E})\cap C_b(X)$, dense in $\mathcal{D}(\mathcal{E})$ and in $L^1(X,\mu)$, and such that for any $f\in \mathcal{C}$ we have $\Gamma(f)\in L^\infty(X,\mu)$ and (\ref{E:Gammatinfty}), then also Assumption \ref{A:C} holds for the kernels $j_\alpha(x,dy)$, $0<\alpha<1$. 

If both is true, then (\ref{E:diagonalize}) holds weakly in $L^1(X,\mu)$ with $\lim_{\overline{\theta}\in \Theta}$ and $j_{p,\overline{\theta}}(x_0,\cdot)$ replaced by $\lim_{\overline{\alpha}\to 1}$ and $j_{p,\overline{\alpha}}(x_0,\cdot)$ as in (\ref{E:fractionalkernelsp}).
\end{corollary}

Given $f\in \mathcal{C}$ and $0<\alpha<1$, we set
\[\Gamma_\alpha(f)(x):=\int_X(f(x)-f(y))^2j_\alpha(x,dy),\quad x\in X.\]
Note that $\Gamma_\alpha(f)=\frac{\alpha}{\mathbb{\Gamma}(1-\alpha)}\int_0^\infty \Gamma_t(f)\frac{dt}{t^\alpha}$ with $\Gamma_t(f)$ as defined in (\ref{E:Gammat}); we abuse notation and distinguish the two objects according to their subscript index $t$ or $\alpha$.

\begin{proof}
To see that Assumption \ref{A:kernels} holds, note first that (\ref{E:symmetry})is again straightforward from the symmetry of $(P_t)_{t>0}$. Using Fubini, (\ref{E:spectral}) and 
\[\int_X(f(x)-f(y))^2P_t(x,dy)=f(x)^2-2f(x)P_tf(x)+P_t(f^2)(x),\]
we find that 
\begin{align}
\int_X\Gamma_\alpha(f)\:d\mu&=\frac{\alpha}{2\mathbb{\Gamma}(1-\alpha)}\int_0^1\int_X\int_X(f(x)-f(y))^2P_t(x,dy)\mu(dx)\frac{dt}{t^{\alpha+1}}\notag\\
&\hspace{50pt}+\frac{\alpha}{2\mathbb{\Gamma}(1-\alpha)}\int_1^\infty\int_X\int_X(f(x)-f(y))^2P_t(x,dy)\mu(dx)\frac{dt}{t^{\alpha+1}}\label{E:split}\\
&\leq \frac{\alpha}{\mathbb{\Gamma}(2-\alpha)}\:\mathcal{E}(f)+\frac{2}{\mathbb{\Gamma}(1-\alpha)}\|f\|_{L^2(X,\mu)}^2.\notag
\end{align}
To show that 
\begin{equation}\label{E:targetalpha}
\lim_{\alpha\to 1}\int_X\Gamma_\alpha(f)d\mu=\mathcal{E}(f)
\end{equation}
as required in (\ref{E:Gammathetalim}), let $0<\varepsilon<1$. Choose $t_\varepsilon>0$ such that 
\[\mathcal{E}(f)-\int_X\Gamma_t(f)\:d\mu<\frac{\varepsilon}{8}.\]
Then for any $0<\alpha<1$ we have 
\[\left|(1-\alpha)t_\varepsilon^{\alpha-1}\int_0^{t_\varepsilon}\int_X\Gamma_t(f)\:d\mu\frac{dt}{t^\alpha}-\int_X\Gamma(f)\:d\mu\right|\leq (1-\alpha)t_\varepsilon^{\alpha-1}\int_0^{t_\varepsilon}\left|\int_X\Gamma_t(f)\:d\mu-\int_X\Gamma(f)\:d\mu\right|\frac{dt}{t^\alpha}<\frac{\varepsilon}{8}.\]
Since 
$\lim_{\alpha\to 1}\frac{\alpha t_\varepsilon^{1-\alpha}}{\mathbb{\Gamma}(2-\alpha)}=1$,
we can choose $\alpha$ close enough to $1$ to have 
\[\left|\frac{\alpha t_\varepsilon^{1-\alpha}}{\mathbb{\Gamma}(2-\alpha)}-1\right|<\frac{\varepsilon}{2(1+\mathcal{E}(f))}.\]
Estimating similarly as in (\ref{E:split}), we can choose $\alpha$ so that also  
\[(1-\alpha)t_\varepsilon^{\alpha-1}\int_{t_\varepsilon}^\infty\int_X\Gamma_t(f)\:d\mu\frac{dt}{t^\alpha}\leq \frac{2(1-\alpha)}{\alpha}t_\varepsilon^{-1}\|f\|_{L^2(X,\mu)}^2<\frac{\varepsilon}{8}\]
holds. Combining these estimates gives
\[\left|\mathcal{E}(f)-\int_X\Gamma_\alpha(f)\:d\mu\right|\leq \left|1-\frac{\alpha t_\varepsilon^{1-\alpha}}{\mathbb{\Gamma}(2-\alpha)}\right|\mathcal{E}(f)+\frac{\alpha t_\varepsilon^{1-\alpha}}{\mathbb{\Gamma}(2-\alpha)}\left|(1-\alpha)t_\varepsilon^{\alpha-1}\int_0^\infty \int_X \Gamma_t(f)\:d\mu\frac{dt}{t^\alpha}-\int_X\Gamma(f)\:d\mu\right|<\varepsilon,\]
which shows (\ref{E:targetalpha}). To see (\ref{E:Gammathetainfty}), we can proceed similarly as in (\ref{E:split}) and find that 
\begin{align}
\Gamma_\alpha(f)(x)&=\frac{\alpha}{\mathbb{\Gamma}(1-\alpha)}\int_0^1\Gamma_t(f)\frac{dt}{t^{\alpha}}+\frac{\alpha}{2\mathbb{\Gamma}(1-\alpha)}\int_1^\infty\int_X(f(x)-f(y))^2P_t(x,dy)\frac{dt}{t^{\alpha+1}}\notag\\
&\leq \frac{\alpha}{\mathbb{\Gamma}(2-\alpha)}\:\sup_{0<t<1}\Gamma_t(f)(x)+\frac{2\left\|f\right\|_{\sup}^2}{\mathbb{\Gamma}(1-\alpha)}.\notag
\end{align}
\end{proof}

\appendix

\section{Estimates for wedge products}\label{App:wedge}

We give a proof Lemma \ref{L:wedgewelldef}, the arguments are standard.
\begin{proof}
Item (i) is a variant of a standard formula. To verify it, let $\{e_i:\ i=1,2,...\}$ be an orthonormal basis of $T_x^\ast X$. Then 
\begin{equation}\label{E:ortho}
\{e_{i_1}\wedge ...\wedge e_{i_p}:\ i_k \in \{1,2,...\},\ i_1<...<i_p\},
\end{equation}
is an orthonormal basis of $\hat{\Lambda}^pT_x^\ast X$, see for instance \cite[p. 338]{Temam97}. Similarly for $q$ or $p+q$ in place of $p$. If $v=\sum_{i_1<...<i_p} v_{i_1\cdots i_p} e_{i_1}\wedge ...\wedge e_{i_p}$ and $w=\sum_{j_1<...<j_q} w_{j_1\cdots j_q} e_{j_1}\wedge ...\wedge e_{j_q}$ are finite linear combinations with real coefficients $v_{i_1\cdots i_p}$ and $w_{j_1\cdots j_q}$, then 
\[v\wedge w=\sum_{i_1<...<i_p} \sum_{j_1<...<j_q} v_{i_1\cdots i_p} w_{j_1\cdots j_q} e_{i_1}\wedge ...\wedge e_{i_p}\wedge e_{j_1}\wedge ...\wedge e_{j_q}.\]
A summand for which the indices $i_1,...,i_p,j_1,...,j_q$ are not all different is zero. If all are different, then there are $(p+q)!/(p!q!)$ different ways (shuffles) to arrange them so that $i_1<...<i_p$ and $j_1<...<j_q$, and this number is an upper bound for the number of summands that can -- up to sign -- contain the unique element of the orthonormal basis of $\hat{\Lambda}^{p+q}T_x^\ast X$
corresponding to these indices. Consequently 
\[\left\|v\wedge w\right\|_{\hat{\Lambda}^{p+q}T_x^\ast X}^2\leq \Big(\frac{(p+q)!}{p!q!}\Big)^2\sum_{i_1<...<i_p} \sum_{j_1<...<j_q} v_{i_1\cdots i_p}^2 w_{j_1\cdots j_q}^2=\Big(\frac{(p+q)!}{p!q!}\Big)^2 \left\|v\right\|_{\Lambda^{p}T_x^\ast X}^2\left\|w\right\|_{\Lambda^{q}T_x^\ast X}^2.\]
For (ii), let $\{\xi_i:\ i=1,2,...\}\subset \mathcal{M}(T^\ast X)$ be a measurable field of orthonormal bases for $T^\ast X$. Given $i_1<...<i_p$ we then have $\xi_{i_1}\wedge ...\wedge \xi_{i_p}\in \mathcal{M}^p(T^\ast X)$, because 
\[\left\langle \xi_{i_1,x}\wedge ... \wedge \xi_{i_p,x},d_{0,x}f_1\wedge ... \wedge d_{0,x}f_p\right\rangle_{\Lambda^pT_x^\ast X}=\det[(\left\langle \xi_{i_k,x},d_{0,x}f_\ell\right\rangle_{T^\ast X} )_{k,\ell=1}^\infty)]\]
for any $f_1,...,f_p\in \mathcal{A}$ and $x\in X$, and this is measurable. The elements $\xi_{i_1,x}\wedge ... \wedge \xi_{i_p,x}$, $i_1<...<i_p$, form an orthonormal basis of $\hat{\Lambda}^pT_x^\ast X$. Given an element $\omega=(\omega_x)_{x\in X}$ of $\mathcal{M}^p(T^\ast X)$, its evaluation $\omega_x$ at $x$ admits the representation 
\[\omega_x=\sum_{i_1<...<i_p}\omega_{i_1\cdots i_p}(x) \xi_{i_1,x}\wedge ... \wedge \xi_{i_p,x}\]
with coefficients $\omega_{i_1\cdots i_p}(x):=\left\langle \omega_x, \xi_{i_1,x}\wedge ... \wedge \xi_{i_p,x}\right\rangle_{\Lambda^pT_x^\ast X}$, which by the preceding are measurable functions of $x\in X$. Given an element $\eta=(\eta_x)_{x\in X}$ of $\mathcal{M}^q(T^\ast X)$, the same argument gives 
\[\eta_x=\sum_{j_1<...<j_q}\eta_{j_1\cdots j_q}(x)\xi_{j_1,x}\wedge ... \wedge \xi_{j_q,x}\] 
with measurable coefficients $\eta_{j_1\cdots j_q}$. It follows that for any $f_1,...,f_p,f_{p+1},...,f_{p+q}\in \mathcal{A}$ the function 
\begin{multline}
x\mapsto \left\langle \omega_x\wedge \eta_x,d_{0,x}f_1\wedge ... \wedge d_{0,x}f_{p+q}\right\rangle_{\Lambda^{p+q}T_x^\ast X}\notag\\
=\sum_{i_1<...<i_p}\sum_{j_1<...<j_q}\omega_{i_1\cdots i_p}(x)\eta_{j_1\cdots j_q}(x)\times \notag \\
\times \left\langle \xi_{i_1,x}\wedge ... \wedge \xi_{i_p,x}\wedge \xi_{j_{p+1},x}\wedge ... \wedge \xi_{j_{p+q},x}, d_{0,x}f_1\wedge ... \wedge d_{0,x}f_{p+q}\right\rangle_{\Lambda^{p+q}T_x^\ast X}
\end{multline}
is measurable. Item (iii) is immediate from (ii) and (\ref{E:wedgebound}), and (iv) is easily seen. 
\end{proof}

\section{Comments on change of measure}\label{App:change}

For the convenience of the reader we provide a short proof of Theorem \ref{T:changemeasure}. 

\begin{proof} By the separability of $X$ we can find a sequence $(f_k)_{k\geq 1}\subset C_c(X)\cap\mathcal{D}(\mathcal{E})$ with span dense both in $\mathcal{D}(\mathcal{E})$ and $C_c(X)$, a suitable weighted sum of their energy measures gives $\mu'$, see \cite[Lemma 2.1]{HKT15} for a detailed proof. Now let $\mathcal{C}$ be the algebra of all functions obtained from the $f_k$ by taking linear combinations, products and compositions with Lipschitz functions $F:\mathbb{R}\to\mathbb{R}$ such that $F(0)=0$. A suitable mechanism to show the closability of $(\mathcal{E},\mathcal{C})$ on $L^2(X,\mu')$ was formulated in \cite[Th\'eor\`eme 9]{Mokobodzki95}, a variant of it was studied further in \cite{H16}: The form $(\mathcal{E},\mathcal{C})$ is closable with respect to the supremum norm, \cite[Theorem 2.1]{H16}. Since for any compact $K\subset X$ and relatively compact open $U$ containing $X$ we can find $\varphi\in \mathcal{C}$ with $0\leq \varphi\leq 1$, one on $K$ and zero outside $U$, we can find a strictly positive function $\chi\in \mathcal{C}$. Since $\mu'$ dominates all energy measures of functions from $\mathcal{C}$, \cite[Theorem 2.2]{H16} and the comments following it give the closability of $(\mathcal{E},\mathcal{C})$ in $L^2(X,\mu')$. The Markov property of its closure $(\mathcal{E}',\mathcal{D}(\mathcal{E}'))$ can be verified using \cite[Theorem 1.4.2 (v)]{FOT94}. To see the locality of $(\mathcal{E}',\mathcal{D}(\mathcal{E}'))$, we can use (\ref{E:chainrule}) similarly as in \cite[Theorem 6.1.2 and its proof]{BH91} and conclude that if $F\in C_c^1(\mathbb{R})$, $f\in \mathcal{D}(\mathcal{E}')$ and $(f_n)_{n\geq 1}\subset \mathcal{C}$ converges to $f$ in $\mathcal{D}(\mathcal{E}')$, then $(F(f_n)-F(0))_{n\geq 1}\subset \mathcal{C}$ is Cauchy in $\mathcal{D}(\mathcal{E}')$ with limit $F(f)-F(0)$. Now polarization shows that if $F,G\in C_c^\infty(\mathbb{R})$ have disjoint supports and $f$ and $f_n$ are as before, we have 
\[\mathcal{E}'(F(f)-F(0),G(f)-G(0))=\lim_{n\to\infty} \mathcal{E}(F(f_n)-F(0),G(f_n)-G(0))=0.\]
Alternatively, the locality can be seen similarly as in the proof of \cite[Theorem 3.1.2]{FOT94}.
\end{proof}


\begin{thebibliography}
\normalsize
\bibitem{AHM88}
R. Abraham, J. E. Marsden, T. Ratiu, \emph{Manifolds, Tensor Analysis, and Applications}, Appl. Math Sci. vol. 75, 2nd ed., Springer, New York, 1988.

\bibitem{Alexander35}
J. W. Alexander, \emph{On the ring of a compact metric space}, Proc. Acad. Sci. {\bf 21} (1935), 509--512.

\bibitem{AGS14}
L. Ambrosio, N. Gigli, G. Savar\'e, \emph{Metric measure spaces with Riemannian Ricci curvature bounded from
below}, Duke Math. J. {\bf 163} (7) (2014), 1405--1490.

\bibitem{Arnold18}
D. N. Arnold, \emph{Finite Element Exterior Calculus}, CBMS-NSF Regional Conf. Series in Appl. Math., SIAM, Philadelphia, 2018.

\bibitem{BGL14} 
D. Bakry, I. Gentil, M. Ledoux, \emph{Analysis and Geometry of Markov Diffusion Operators}, Grundlehren math. Wiss. 348, Springer International, 2014.


\bibitem{BSSS12}
L. Bartholdi, T. Schick, N. Smale, S. Smale, \emph{Hodge theory on metric spaces},
Found. Comp. Math. {\bf 12} (2012), 1--48.


\bibitem{BT82}
R. Bott, L. W. Tu, \emph{Differential Forms in Algebraic Topology}, Grad. Texts in Math. vol. 82, Springer, Berlin, 1982. 

\bibitem{BH91}
N. Bouleau, F. Hirsch, \emph{Dirichlet Forms and Analysis on Wiener Space},
deGruyter Studies in Math. 14, deGruyter, Berlin, 1991.


\bibitem{BBM01}
J. Bourgain, H. Brezis, P. Mironescu, \emph{Another look at Sobolev spaces}, in: Optimal Control and Partial Differential Equations - Innovations and Applications, IOS Publishers, Amsterdam, 2001, pp. 439--455.

\bibitem{Braun21}
M. Braun, \emph{Vector calculus for tamed Dirichlet spaces}, preprint (2021), arXiv:2108:12374.

\bibitem{BL92}
J. Br\"uning, M. Lesch, \emph{Hilbert complexes},
J. Funct. Anal. {\bf 108} (1992), 88-132.

\bibitem{BuragoIvanovKurylev14}
D. Burago, S. Ivanov, Y. Kurylev, \emph{A graph discretization of the Laplace-Beltrami operator},
J. Spectr. Theory {\bf 4} (2014), 675--714.

\bibitem{CKS87}
E. Carlen, S. Kusuoka, D. Stroock, \emph{Upper bounds for symmetric Markov transition functions},
Ann. Inst. H. Poinc. Probab. Stat. {\bf 23}(2) (1987), 245--287.

\bibitem{Carlsson}
G. Carlsson, \emph{Topology and Data}, Bull. Amer. Math. Soc. {\bf 46} (2) (2009), 255--308.

\bibitem{ChazaldeSilvaOudot}
F. Chazal, V. deSilva, S. Oudot, \emph{Persistence stability for geometric complexes},
Geom. Dedicata {\bf 173} (2014), 193--214.

\bibitem{Cheeger}
J. Cheeger, \emph{On the Hodge theory of Riemannian pseudomanifolds}, In: Geometry of the
Laplace operator, Proc. Sympos. Pure Math. Vol. 36, Amer. Math. Soc., Providence, 1980, pp. 91--146.

\bibitem{Ch99}
J. Cheeger, \emph{Differentiability of Lipschitz functions on metric measure spaces}, Geom.
Funct. Anal. {\bf 9} (1999), 428--517.

\bibitem{CGIS13}
F. Cipriani, D. Guido, T. Isola, J.-L. Sauvageot, \emph{Integrals and potentials of differential $1$-forms in the Sierpinski gasket}, Adv. Math. {\bf 239} (2013), 128--163.

\bibitem{CS03}
F. Cipriani, J.-L. Sauvageot, \emph{Derivations as square roots of Dirichlet forms},
J. Funct. Anal. {\bf201} (2003), 78--120.

\bibitem{CS09}
F. Cipriani, J.-L. Sauvageot, \emph{Fredholm modules on p.c.f. self-similar fractals and their conformal geometry},
Comm. Math. Phys. {\bf286} (2009), 541-558.

\bibitem{CM90}
A. Connes, H. Moscovici, \emph{Cyclic cohomology, the Novikov conjecture and hyperbolic groups},
Topology {\bf 29} (3) (1990), 345--388.


\bibitem{Da89}
E. B. Davies, \emph{Heat Kernels and Spectral Theory}, Cambridge Tracts in Math. 92, Cambridge Univ. Press, Cambridge, 1989.

\bibitem{dRh60}
G. deRham, \emph{Vari\'et\'es Differentiables}, Hermann, 1960, Paris.

\bibitem{dRh-K50}
G. deRham, K. Kodaira, \emph{Harmonic Integrals}, Inst. Adv. Study, Princeton, 1950.


\bibitem{Dixmier81}
J. Dixmier, \emph{Von Neumann Algebras},
North-Holland Math. Lib. 27, North-Holland, Amsterdam, 1981.

\bibitem{Dodziuk76}
J. Dodziuk, \emph{Finite-difference approach to the Hodge theory of harmonic forms}, Amer. J. Math. {\bf 98} (1) (1976), 79--104.

\bibitem{Dodziuk81}
J. Dodziuk, \emph{Sobolev spaces of differential forms and deRham-Hodge isomorphism}, J. Diff. Geo. {\bf 16} (1981), 63--73.

\bibitem{Eberle99}
A. Eberle, \emph{Uniqueness and Non-uniqueness of Semigroups Generated by Singular Diffusion Operators}, Springer Lect. Notes Math. 1718, Springer, New York, 1999.

\bibitem{ELZ}
H. Edelsbrunner, D. Letscher, A. Zomorodian, \emph{Topological persistence and simplification}, Discr.
Comput. Geom. {\bf 28} (4) (2002), 511--533.

\bibitem{Ei99}
D. Eisenbud, \emph{Commutative Algebra: with a View Toward Algebraic Geometry},
Grad. Texts in Math. 150, Springer, New York, 1999.

\bibitem{ERSZ06}
A.F.M. ter Elst, D. W. Robinson, A. Sikora , Y. Zhu, \emph{Dirichlet Forms and
Degenerate Elliptic Operators}, In: Operator Theory:
Advances and Applications, Vol. 168, Birkh\"auser, Basel, 2006, pp. 73--95.

\bibitem{ERSZ07}
A.F.M. ter Elst, D. W. Robinson, A. Sikora , Y. Zhu, \emph{Second-order operators with degenerate coefficients}, Proc. London Math. Soc. {\bf 95} (3) (2007), 299--328.

\bibitem{EKS15} 
M. Erbar, K. Kuwada, K.-Th. Sturm, \emph{On the equivalence of the entropic curvature-dimension condition and Bochner's inequality on metric measure spaces}, Invent. Math. {\bf 201} (2015), 993--1071.

\bibitem{FukushimaLeJan91}
M. Fukushima, Y. LeJan, \emph{On quasi-supports of smooth measures and closability of pre-Dirichlet forms},
Osaka J. Math. {\bf 28} (1991), 837--845.

\bibitem{FOT94}
M. Fukushima, Y. Oshima and M. Takeda, \emph{Dirichlet Forms and Symmetric Markov Processes},
deGruyter, Berlin, New York, 1994.

\bibitem{FukushimaSatoTaniguchi91}
M. Fukushima, K. Sato, S. Taniguchi, \emph{On the closable parts of pre-Dirichlet forms and the fine supports
of underlying measures}, Osaka J. Math. {\bf 28} (1991), 517--535.

\bibitem{Gaffney55}
M. P. Gaffney, \emph{Hilbert space methods in the theory of harmonic integrals},
Trans. Amer. Math. Soc. {\bf 78} (2) (1955), 426--444.

\bibitem{Genton}
L. Genton, \emph{Scaled Alexander-Spanier Cohomology and Lqp Cohomology
for Metric Spaces}, TH\`ESE No. 6330 (2014), PhD-thesis, EPFL, Lausanne, 2014.

\bibitem{Gigli17}
N. Gigli, \emph{Non-smooth differential geometry - an approach tailored for spaces with Ricci curvature bounded from below}, Mem. Amer. Math. Soc. {\bf 251} (11) (2017). 

\bibitem{Gilkey74}
P. B. Gilkey, \emph{The Index Theorem and the Heat Equation}, Publish or Perish,
Boston MA, 1974.


\bibitem{GVF}
J. Gracia-Bondia, J. V\'arilly, H. Figueroa, \emph{Elements of Noncommutative Geometry}, Birkh\"auser, 2001.

\bibitem{Grigoryan2009}
A. Grigor'yan, \emph{Heat Kernel and Analysis on Manifolds}, AMS/IP Studies in Advanced Mathematics 47, Boston, MA, 2009.

\bibitem{Gromov87}
M. Gromov, \emph{Hyperbolic groups}, In: Essays in Group Theory, Math. Sci. Res. Inst. Pub. vol. 8, Springer, New York, 1987, pp. 75--263.

\bibitem{Hausmann95}
J.-C. Hausmann, \emph{On the Vietoris-Rips complexes and a cohomology theory for metric
spaces}, Prospects in Topology: Proceedings of a Conference in honour of William
Browder, Princeton, 1995, pp. 175--188.

\bibitem{Hino10}
M. Hino, \emph{Energy measures and indices of Dirichlet forms, with applications to derivatives on some fractals},
Proc. London Math. Soc. {\bf100}  (2010), 269-302.

\bibitem{H15}
M. Hinz, \emph{Magnetic energies and Feynman-Kac-It\^o formulas for symmetric Markov processes}, 
Stoch. Anal. Appl. {\bf 33}(6) (2015), 1020-1049.

\bibitem{H16}
M. Hinz, \emph{Sup-norm closable bilinear forms and Lagrangians}, 
Ann. Mat. Pura Appl. {\bf 195} (4) (2016), 1021--1054.

\bibitem{HKT15}
M. Hinz, D. Kelleher, A. Teplyaev, \emph{Metrics and spectral triples for Dirichlet
and resistance forms}, J. Noncomm. Geom. {\bf 9} (2) (2015), 359--390.

\bibitem{HKM20}
M. Hinz, D. Koch, M. Meinert, \emph{Sobolev spaces and calculus of variations on fractals}, in: Analysis, Probability and Mathematical Physics on Fractals, World Scientific, 2020, pp. 419--450. 

\bibitem{HinzKommer}
M. Hinz, J. Kommer, \emph{A tensor product approach to non-local differential complexes},
Math. Ann. {\bf 389} (2024), 2357--2409.

\bibitem{HMS23}
M. Hinz, J. Masamune, K. Suzuki, \emph{Removable sets and Lp-uniqueness on manifolds and metric measure spaces}, Nonlinear Anal. {\bf 234} (2023), 113296.


\bibitem{HRT13}
M. Hinz, M. R\"ockner, A. Teplyaev, \emph{Vector analysis for  Dirichlet forms and quasilinear PDE and SPDE on metric measure spaces}, Stoch. Proc. Appl. {\bf 123} (12) (2013), 4373--4406.

\bibitem{HR16}
M. Hinz, L. Rogers, \emph{Magnetic fields on resistance spaces},  J. Fractal Geometry {\bf 3} (2016),  75--93.

\bibitem{HT15}
M. Hinz,   A. Teplyaev, \emph{Local Dirichlet forms, Hodge theory, and
the Navier-Stokes equations on topologically one-dimensional
fractals}, Trans. Amer. Math. Soc. {\bf 367}  (2015),
1347--1380, Corrigendum in Trans. Amer. Math. Soc. {\bf 369} (2017), 6777-6778.

\bibitem{HinzTeplyaev18}
M. Hinz,  A. Teplyaev, \emph{Densely defined non-closable curl on carpet like metric measure spaces}, Math. Nachr. {\bf 291} (11-12) (2018), 1743--1756.

\bibitem{Hodge41}
W. V. D. Hodge, \emph{The Theory and Applications of Harmonic Integrals}, Cambridge Univ.
Press, Cambridge, 1941.

\bibitem{IRT12}
M. Ionescu, L. Rogers, A. Teplyaev, \emph{Derivations, Dirichlet forms and spectral analysis},
J. Funct. Anal. {\bf 263} (2012), no. 8, 2141--2169. 


\bibitem{Jacob01}
N. Jacob, \emph{Pseudo-Differential Operators and Markov processes. Volume I: Fourier Analysis and Semigroups}, Imperial College Press, London, 2001.

\bibitem{Kelley75}
J. L. Kelley, \emph{General Topology}, Springer, New York, 1975.

\bibitem{Ki89}
J. Kigami, \emph{A harmonic calculus on the Sierpi\'nski spaces}, Japan J. Appl. Math. {\bf 6} (1989),
259--290.

\bibitem{Ki90}
J. Kigami, \emph{Harmonic metric and Dirichlet form on the Sierpi\'nski gasket}, Asymptotic prob-
lems in probability theory: stochastic models and diffusions on fractals (Sanda/Kyoto,
1990), 201--218, Pitman Res. Notes Math. Ser., 283, Longman Sci. Tech., Harlow, 1993.

\bibitem{Ki01}
J. Kigami, \emph{Analysis on Fractals}, Cambridge Univ. Press, Cambridge, 2001.


\bibitem{Kodaira49}
K. Kodaira, \emph{Harmonic fields in Riemannian manifolds (generalized potential theory)}, Ann. Math. {\bf 50} (1949), 587--665.

\bibitem{Kolmogorov36}
A. N. Kolmogorov, \emph{\"Uber die Dualitat im Aufbau der
kombinatorischen Topologie}, Rec. Math. (Mat. Sbornik) N.S. {\bf 1} (43) (1) (1936), 97--102.

\bibitem{Kolmogorov36a}
A. N. Kolmogorov, \emph{Homologiering des Komplexes und des lokal-bikompakten Raumes}, 
Rec. Math. (Mat. Sbornik) N.S. {\bf 1} (43) (5) (1936), 701--706.

\bibitem{Kommer}
J. Kommer, \emph{A tensor product approach to non-local differential complexes}, 
PhD thesis, Bielefeld University, Bielefeld, 2023.

\bibitem{KoskelaZhou}
P. Koskela, Y. Zhou, \emph{Geometry and analysis of Dirichlet forms}, Adv. Math. {\bf 231} (2012), 2755--2801.

\bibitem{Kusuoka89}
S. Kusuoka, \emph{Dirichlet forms on fractals and products of random matrices}, Publ.
Res. Inst. Math. Sci. {\bf 25} (1989), 659--680.

\bibitem{KuwaeNakao91}
K. Kuwae, Sh. Nakao, \emph{Time changes in Dirichlet space theory}, Osaka J. Math. {\bf 28} (1991), 847--865.

\bibitem{Latschev01}
J. Latschev, \emph{Vietoris-Rips complexes of metric spaces near a closed Riemannian manifold},
Arch. Math. {\bf 77} (2001), 522--528.


\bibitem{LeJan78}
Y. LeJan, \emph{Mesures associ\'ees \`a une forme de Dirichlet. Applications.}, Bull. Soc. Math. France {\bf 106}
(1978), 61--112.

\bibitem{Lueck}
W. L\"uck, \emph{$L^2$-Invariants: Theory and Applications to Geometry and K-Theory}, Ergebnisse der Mathematik und ihrer Grenzgebiete 44, Springer, Berlin, 2002.

\bibitem{MaR92}
Z. M. Ma, M. R\"ockner, \emph{Introduction to the Theory of (Non-Symmetric) Dirichlet Forms}, Universitext, Springer, Berlin, 1992.

\bibitem{Massey78}
W. M. Massey, \emph{Homology and Cohomology Theory. An approach based on Alexander-Spanier cochains}, Monographs and textbooks in pure and applied mathematics, vol. 46, Marcel Dekker, Inc., New York, 1978.

\bibitem{Mokobodzki95}
G. Mokobodzki, \emph{Fermeabilit\'e des formes du Dirichlet et in\'egalit\'e de type Poincar\'e}, Pot. Anal. {\bf 4} (1995), 409--413.

\bibitem{Nakao85}
Sh. Nakao, \emph{Stochastic calculus for continuous additive functionals of zero energy}, Z. Wahrsch. verw. Gebiete {\bf 68} (1985), 557--578.

\bibitem{Pansu96}
P. Pansu, \emph{Introduction to $L^2$ Betti numbers}, In: Riemannian geometry, Waterloo, 1993. Fields
Inst. Monogr. 4, Amer. Math. Soc., Providence, 1996, pp. 53--86.

\bibitem{Pansu04}
P. Pansu, \emph{Cohomologie $L^p$: invariance sous quasiisom\'etries}, preprint (2004).

\bibitem{Rosenberg}
S. Rosenberg, \emph{The Laplacian on a Riemannian manifold}, London Math.
Soc. Student Texts vol. 31, Cambridge University Press, Cambridge, 1997.

\bibitem{Sato99}
K.-I. Sato, \emph{L\'evy Processes and Infinitely Divisible Distributions}, Cambridge studies in adv. math. 68, Cambridge Univ. Press, Cambridge, 1999.

\bibitem{Sauvageot89}
J.-L. Sauvageot, \emph{Tangent bimodule and locality for dissipative operators on $C^\ast$-algebras},
Quantum Probability and Applications IV, Lect. Notes Math. 1396, Springer, Berlin, 1989, pp. 322--338.

\bibitem{Sauvageot90}
J.-L. Sauvageot, \emph{Quantum differential forms, differential calculus and semi-groups}, Quantum
Probability and Applications V, Lect. Notes Math. 1442, Springer, Berlin, 1990,
pp. 334--346.

\bibitem{Sh00}
N. Shanmugalingam, \emph{Newtonian spaces: an extension of Sobolev spaces to metric
measure spaces}, Rev. Mat. Iberoamericana {\bf 16} (2000), 243--279.

\bibitem{SS12}
N. Smale, S. Smale, \emph{Abstract and classical Hodge DeRham theory},
Anal. Appl. {\bf 10}(1) (2012), 91--111.

\bibitem{Spanier48}
E. H. Spanier, \emph{Cohomology for general spaces}, Ann. Math. {\bf 49} (2) (1948), 407--427.

\bibitem{Spanier66}
E. H. Spanier, \emph{Algebraic Topology}, Springer, New York, 1966.

\bibitem{Strichartz05}
R. S. Strichartz, \emph{Analysis on products of fractals}, Trans. Amer.
Math. Soc. {\bf 357} (2005), 571--615.

\bibitem{Str06}
R. S. Strichartz, \emph{Differential Equations on Fractals: A Tutorial}, Princeton Univ. Press, Princeton 2006.

\bibitem{Takesaki79}
M. Takesaki, \emph{Theory of Operator Algebras}, Springer, New York, 1979.

\bibitem{Taylor}
M. E. Taylor, \emph{Partial Differential Equations I. Basic Theory}, Appl. Math. Sci. vol. 115, 2nd ed., Springer, New York, 2011.  

\bibitem{Temam97}
R. Temam, \emph{Infinite Dimensional Dynamical Systems in Mechanics and Physics},
Springer, New York, 1997.

\bibitem{T08}
A. Teplyaev, \emph{Harmonic coordinates on fractals with finitely ramified cell structure}, Canad. J. Math. {\bf 60} (2008), 457--480.

\bibitem{Vietoris27}
L. Vietoris, \emph{\"Uber den h\"oheren Zusammenhang kompakter R\"aume und eine Klasse von zusammenh\"angenden Abbildungen}, Math. Ann. {\bf 91} (1927), 454--472.

\bibitem{Warner}
F.W. Warner, \emph{Foundations of Differentiable Manifolds and Lie Groups}, Grad. Texts in Math. 94, Springer, New York, 1983.

\bibitem{W00}
N. Weaver, \emph{Lipschitz algebras and derivations II. Exterior differentiation}, J. Funct. Anal. {\bf 178} (2000), 64-112.


\end{thebibliography}
\end{document}